\documentclass[reqno,11pt]{amsart}
\usepackage{amssymb}
\usepackage{graphicx}
\usepackage{amsmath}
\usepackage[margin=1.23in]{geometry}

\usepackage{esint}
\usepackage{cite}
\usepackage[nobysame]{amsrefs}
\usepackage{appendix}

\usepackage[usenames, dvipsnames]{color}
\usepackage{verbatim}
\usepackage{mathrsfs}
\usepackage{bm}
\usepackage{cite}
\usepackage{todonotes}
\allowdisplaybreaks[4]

\numberwithin{equation}{section}

\newtheorem{theorem}{Theorem}[section]
\newtheorem{corollary}[theorem]{Corollary}
\newtheorem{lemma}[theorem]{Lemma}
\newtheorem{prop}[theorem]{Proposition}

\theoremstyle{definition}
\newtheorem{remark}[theorem]{Remark}

\theoremstyle{definition}

\theoremstyle{definition}

\makeatletter
\def\dashint{\operatorname%
	{\,\,\text{\bf-}\kern-.80em\DOTSI\intop\ilimits@\!\!}}
\makeatother

\def\\det{\text{\det}}

\def\R{\mathbb{R}}

\def\.5{\frac{1}{2}}

\def\bR{\mathbb{R}}

\def\bS{\mathbb{S}}

\newcommand{\RN}[1]{%
	\textup{\uppercase\expandafter{\romannumeral#1}}%
}

\renewcommand{\epsilon}{\varepsilon}
\newcommand{\osc}{\operatorname{osc}}
\newcommand{\dv}{\operatorname{div}}
\newcounter{marnote}

\begin{document}


\title[Boundary estimates for the insulated problem]{Optimal boundary gradient estimates  for the insulated conductivity problem}

\author[H.G. Li]{Haigang Li}
\address[H.G. Li]{School of Mathematical Sciences, Beijing Normal University, Laboratory of MathematiCs and Complex Systems, Ministry of Education, Beijing 100875, China.}
\email{hgli@bnu.edu.cn}

\author[Y. Zhao]{Yan Zhao}
\address[Y. Zhao]{School of Mathematical Sciences, Beijing Normal University, Laboratory of MathematiCs and Complex Systems, Ministry of Education, Beijing 100875, China.}
\email{zhaoyan\_9926@mail.bnu.edu.cn}

\maketitle

\begin{abstract}
In this paper we study the boundary gradient estimate of the solution to the insulated conductivity problem with the Neumann boundary data when a convex insulating inclusion approaches the boundary of the matrix domain. The gradient of solutions may blow up as the distance between the inclusion and the boundary, denoted as $\varepsilon$, approaches to zero. The blow up rate was previously known to be sharp in dimension $n=2$ (see Ammari et al.\cite{AKLLL}). However, the sharp rates in dimensions $n\geq3$ are still unknown. In this paper, we solve this problem by establishing upper and lower bounds on the gradient and prove that the optimal blow up rates of the gradient are always of order $\epsilon^{-1/2}$ for general strictly convex inclusions in dimensions $n\geq3$.  Several new difficulties are overcome and the impact of the boundary data on the gradient is specified. This result highlights a significant difference in blow-up rates compared to the interior estimates in recent works (\cites{LY,Weinkove,DLY,DLY2,LZ}), where the optimal rate is $\epsilon^{-1/2+\beta(n)}$, with $\beta(n)\in(0,1/2)$ varying with dimension $n$. Furthermore, we demonstrate that the gradient does not blow up for the corresponding Dirichlet boundary problem.
\end{abstract}

\section{Introduction}

In this paper we study the boundary gradient estimate for the conductivity problem when one insulating inclusion is located very close to the boundary of the matrix domain.  We start by describing the nature of our domain. Let $D$ be a bounded domain in $\mathbb{R}^{n}$ that contains a strictly convex open set $D_{1}$, with a small distance $\varepsilon:= \mbox{dist}(D_1, \partial{D})$ from the boundary. We assume that both $\partial{D}$ and $\partial{D}_{1}$ are of class $C^{2}$, and consider the following two kinds of boundary value problems with prescribed Dirichlet and Neumann boundary data: 
for a given $\varphi\in C^{1,\alpha}(\partial{D})$, $\alpha>0$,
\begin{equation}\label{equinftykD}
\begin{cases}
\nabla\cdot\left(a_{k}(x)\nabla u_{k}\right)=0 &\mbox{in}~ D,\\
u_{k}=\varphi&\mbox{on}~\partial{D},
\end{cases}
\end{equation}
and for a given $\phi\in C^{\alpha}(\partial{D})$, 
\begin{equation}\label{equinftykN}
\begin{cases}
\nabla\cdot\left(a_{k}(x)\nabla u_{k}\right)=0 &\mbox{in}~ D,\\
\partial_{\nu}u_{k}=\phi&\mbox{on}~\partial{D},
\end{cases}
\end{equation}
where
\begin{equation}\label{ak1}
a_{k}(x)=
\begin{cases}
k\in(0,\infty)&\mbox{in}~D_{1},\\
1&\mbox{in}~\Omega:=D\setminus{D}_{1},
\end{cases}
\end{equation}
and $\partial_{\nu}u:=\frac{\partial u}{\partial\nu}$, and $\nu$ represents the outward normal of $\partial{D}$. To ensure the existence of the solution to the Neumann problem \eqref{equinftykN}, we additionally assume that $\int_{\partial D}\phi=0$. The equations above can be regarded as simple models for electric conduction when one inclusion is very close to the matrix boundary. Here $a_k$ represents the conductivity, which can be assumed to be $1$ in the matrix after normalization. The solution $u_k$ represents the voltage potential, and the gradient $\nabla u_{k}$ represents the electric fields. From an engineering point of view, it is crucial to estimate $\nabla u_{k}$ in the narrow regions between adjacent inclusions and between the inclusions and the matrix boundary. This is because these narrow gaps may exhibit a high concentration of extreme electric fields in the high contrast composite materials. In the last two decades, significant progress has been achieved in estimating gradients in the conductivity problem, including both the perfect case ($k=\infty$) and the insulted case $(k=0)$. The analogous problem concerning elastic stress in the context of linear elasticity was numerically investigated by Babu\u{s}ka, et al. \cite{bab}.

Before  investigating the impact of the boundary data on the gradient of $u$, we first review some important progress regarding the interior case where $D$ includes two adjacent inclusions $D_{1}$ and $D_{2}$, separated by a small distance $\varepsilon: = \mbox{dist}(D_1, D_2)$. In this situation, the piecewise constant coefficient can be described as follows: \begin{equation}\label{ak2}
a_{k}(x)=
\begin{cases}
k\in(0,\infty)&\mbox{in}~D_{1}\cup{D}_{2},\\
1&\mbox{in}~\widetilde{\Omega}:=D\setminus\overline{D_{1}\cup{D}_{2}},
\end{cases}
\end{equation}
and $D_{1}\cup D_{2}$ is always assumed to be far away from the boundary $\partial{D}$, specifically $\mathrm{dist}(D_1 \cup D_2, \partial {D}) > c$, for some positive constant $c$. For a finite and strictly positive $k$, Bonnetier and Vogelius \cite{BV} first proved that $|\nabla u_{k}|$ remains bounded for two circular touching disks of comparable radii in dimension two. Li and Vogelius \cite{LV} extended this result to general second order elliptic equations of divergence form with piecewise H\"{o}lder coefficients and general shape of inclusions in all dimensions. Subsequently, Li and Nirenberg \cite{LN} further extended to study the elliptic systems, including the Lam\'e system.

When $k$ degenerates to $\infty$ (perfect conductor) or $0$ (insulator), it has been shown in \cite{Kel} that the gradient of the solutions always becomes large as $\epsilon$ tends to $0$. It is well known that as $k$ goes to $\infty$ in equation \eqref{equinftykD} with coefficient \eqref{ak2}, $u_{k}$ converges to the solution of the perfect conductivity problem:
\begin{equation}\label{perfect1}
\begin{cases}
\Delta u=0 &\mbox{in}~ \tilde{\Omega},\\
\nabla u=0&\mbox{on}~\bar{D}_{i},i=1,2,\\
\int_{\partial{D}_{i}}\partial_{\nu}u=0& i=1,2,\\
u=\varphi&\mbox{on}~\partial{D};
\end{cases}
\end{equation}
while, as $k$ goes to $0$, $u_{k}$ converges to the solution of the insulated conductivity problem:
\begin{equation}\label{insulated1}
\begin{cases}
\Delta u=0 &\mbox{in}~ \tilde{\Omega},\\
\partial_{\nu}u=0&\mbox{on}~\partial{D}_{i},i=1,2,\\
u=\varphi&\mbox{on}~\partial{D},
\end{cases}
\end{equation}
see, e.g., appendix of \cites{BLY1,BLY2} for the derivation. First, in two dimensions, when $D_{1}$ and $D_{2}$ are circular inclusions in $\R^2$, Ammari et al.  \cite{AKLLL} and Ammari, Kang and Lim \cite{AKL} showed that the optimal blow up rate of $\nabla u$ is of order $\epsilon^{-1/2}$ for both the perfect and the insulated cases. Yun \cite{Y1} extended the result to bounded strictly convex smooth inclusions. Bao, Li and Yin \cites{BLY1,BLY2} studied the perfect conductivity problem \eqref{perfect1} for two convex inclusions and proved that the optimal blow up rates of $\nabla u$ are, respectively, $\epsilon^{-1/2}$ in dimension two, $|\epsilon\ln\epsilon|^{-1}$ in dimension three, and $\epsilon^{-1}$ in dimensions $n\geq4$. Lim and Yun \cite{LY} also studied the case of spherical perfect conductors in dimension three.  For further research on the asymptotic behavior of $\nabla u$, refer to works such as Kang, Lim and Yun \cite{KLY1,KLY2}, Bonnetier and Triki \cite{BT2} and other related works \cite{ACKLY,BT1,DL,LLY,LWX,DZ,KL,CY,Gor,L}. For the gradient estimates for the Lam\'e system with hard inclusions, see e.g. \cite{BLL,BLL2,Li,KY}. The boundary estimates of Lam\'{e} systems with partially infinite coefficients were studied in \cite{BLJ} and the boundary estimates of the  perfect conductivity problem were investigated in \cite{LX}, respectively. It is worth pointing out that for the boundary estimate of the perfect conductivity problem, the optimal blow-up rates of $\nabla u$ are the same as the interior case in \cite{BLY1}, which vary with dimension $n$. However, the optimality of the blow up rate for the boundary estimates of the insulated problem when one insulator approaches the matrix boundary in dimensions $n\geq3$ remains unresolved. 

For the interior estimates of the insulated conductivity problem, in addition to the aforementioned results in dimension two \cite{AKLLL,AKL}, Bao, Li and Yin \cite{BLY2} also obtained an upper bound of $\nabla u$ of order $\epsilon^{-1/2}$ for all dimensions $n\geq2$ for problem \eqref{insulated1}. However, for about a decade, the question that  whether this upper bound is sharp in dimensions $n\geq3$ remained open. It was not until recently that Li and Yang \cite{LY} took advantage of an extension method in \cite{BLY2} and a Harnack inequality to improve the upper bound in dimensions $n\geq3$ to be of order $\epsilon^{-1/2+\beta}$ for some $\beta>0$. Subsequently, by using a direct maximum principle argument, Weinkove \cite{Weinkove} established an upper bound of order $\epsilon^{-1/2+\beta(n)}$ with a specific constant $\beta(n)>0$ for $n\geq4$ in the case where $D_{1}$ and $D_{2}$ are both balls. The argument presented in \cite{Weinkove} relies on the facts that the solution is bounded and that the blowup only occurs in the narrow region. However, for the Neumann boundary problem, the boundedness of the solution's oscillation may deteriorate. The optimality in dimensions $n\geq3$ was ultimately proved by Dong, Li and Yang \cite{DLY}, particularly with the explicit $\beta(n)= [-(n-1)+\sqrt{(n-1)^2+4(n-2)}]/4$ when the insulators are balls. The significant breakthrough made in \cite{DLY} relies on a known pointwise upper bound $(\varepsilon+|x'|^2)^{-\frac{1}{2}}$ obained by the extension method in \cite{BLY2}, as well as some results on the degenerate elliptic equation where the analysis is based on the harmonic decomposition and the Moser's iteration technique. When insulators are general strictly convex, difficulties arise due to the breaking of radial symmetry. The optimal blowup rate for strictly convex insulators in dimension $n=3$ is finally proved in \cites{DLY2,LZ}. The insulated conductivity problem with $p$-Laplacian was investigated in \cite{DYZ}. 

Our main objective in this paper is to investigate the boundary estimates as strictly convex inclusions approach the matrix boundary and to clarify the effect on the blow up rate of the gradient from the boundary data. In this regard, Ammari, Kang, Lee, Lee, and Lim \cite{AKLLL} considered the boundary problem by using conformal transform to transform it into a kind of interior problem and proved that in two dimensions the optimal blow up rate of $\nabla u$ is as well of order $\epsilon^{-1/2}$ when the circular conductor is perfect or insulated. At the end of Section 1 in \cite{AKLLL}, they mentioned that ``It seems challenging to obtain similar results in three dimensions. At this moment it is even not clear what the blow-up rate of the gradient would be in three dimensions''.  In this present paper, we focus on this boundary estimate problem and prove that the optimal blow up rate is always of order $\epsilon^{-1/2}$ in all dimensions $n\geq3$ for strictly convex inclusions when the Neumann boundary data is prescribed. We provide upper and lower bounds of $\nabla u$ to support this assertion. This discovery contrasts with the interior results discussed above, where the optimal blow up rates vary with dimensions (refer to \cite{DLY, DLY2,LZ}).  For the Dirichlet problem counterpart, we prove that $\nabla u$ is always bounded and does not blow up.  

We consider the Dirichlet and Neumann problems for the insulated conductivity problem in dimensions $n\geq3$. By taking $k\to0$ in equation \eqref{equinftykD}--\eqref{ak1}, the solution $u_k$ to \eqref{equinftykD} and \eqref{equinftykN} will weakly converge to the solution to \eqref{equinfty-0} and \eqref{equinfty0}, respectively. The derivation of \eqref{equinfty-0} and \eqref{equinfty0} is similar to the Appendix in \cite{BLY2}. For a given $\varphi\in C^{1,\alpha}(\partial{D})$, 
\begin{equation}\label{equinfty-0}
\begin{cases}
\Delta{u}=0& \mbox{in}~\Omega:=D\backslash D_1,\\
\partial_{\nu}u=0&\mbox{on}~\partial{D_1},\\
u=\varphi(x)&\mbox{on}~\partial{D};
\end{cases}
\end{equation}
and for a given $\phi\in C^{\alpha}(\partial{D})$, $\alpha\in(1-{2}/{n},1)$, with $\int_{\partial D}\phi=0$, 
\begin{equation}\label{equinfty0}
\begin{cases}
\Delta{u}=0& \mbox{in}~\Omega,\\
\partial_{\nu}u=0&\mbox{on}~\partial{D_1},\\
\partial_{\nu}u=\phi(x)&\mbox{on}~\partial{D}.
\end{cases}
\end{equation}

The boundary gradient estimate for the insulated conductivity problem \eqref{equinfty0} is closely related to the interior estimate. Our initial ideas are inspired by the reduction argument in \cite{DLY}, which reduces \eqref{equinfty0} to the $n-1$ dimension degenerate elliptic equation. However, it encounters several additional challenges. New ideas are needed to overcome them.

The first challenge we faced was that the extension method used in \cites{BLY2,DLY} strictly required that the Neumann boundary data $\phi$ locally equals to $0$ on the boundary of the narrowest region. To address this, instead of trying to transform \eqref{equinfty0} into an elliptic equation with zero Neumann boundary, we adopt a different strategy presented in detail in Section 3. This strategy is roughly as follows. By the $W^{2,p}$ estimates, we can obtain a pointwise upper bound $(\varepsilon+|x'|^2)^{-1}$ for $\nabla u$. Then by using this upper bound and combining with the $C^{1,\alpha}$ estimate, we can obtain an upper bound estimate $(\varepsilon+|x'|^2)^{-1+\sigma}$ for $\nabla\bar u$, where $\sigma\in(0,\frac{1}{2})$, $\bar u$ is the solution to the $n-1$ dimensional elliptic equation. By establishing appropriate estimates of $\nabla(u-\bar u)$ and employing the bootstrap argument, it becomes possible to achieve a stronger upper bound for $\nabla u$.

Secondly,  the type of degenerate elliptic equation discussed in \cite{DLY} for the interior case does not encompass the degenerate elliptic equation of the boundary case when $|\phi(0')|\sim 1$. Here, $|\phi(0')|\sim 1$ means that $\frac{1}{C}\leq |\phi(0')|\leq C$, where $C$ is a universal constant independent of $\varepsilon$. The theorems and the proofs presented in \cite{DLY} can not be modified to accommodate the boundary case when $|\phi(0')|\sim 1$. Indeed, one can review the Moser's iteration argument of Lemma 2.3 in \cite{DLY}. If the hypothesis $1+\gamma-2s>0$ in Lemma 2.3 is relaxed to allow $1+\gamma-2s\geq0$, the uniform global boundedness result will fail as $\varepsilon\to0$. The hypothesis $1+\gamma-2s=0$ in the interior case corresponds to the requirement $|\phi(0')|\sim 1$ in the boundary case. Although uniform global boundedness is not guaranteed in the boundary case, we can establish a uniform local oscillation estimate for the solution to the degenerate elliptic equation, by using a Harnack inequality and the local estimates near boundary. Besides, the global oscillation estimate is established as an iteration result of the local boundary estimate. With these local and global properties, we complete the proof. 

Before stating our main results, we first describe the nature of our domains. Let $D_1^{*}\subset D$ and suppose that $\partial D$ and $\partial D_1^*$ are relatively strictly convex. The boundary of $D_1^{*}$ touches $\partial D$  at the origin, with the inner normal of $\partial D$ being the positive $x_n$-axis . By translating $D_1^{*}$ by $\varepsilon$ along $x_n$-axis, we obtain $D_1=D_1^{*}+(0',\varepsilon)$. We use the notation $x=(x',x_n)$ to represent a point in $\R^{n}$, where $x'\in \R^{n-1}$. Because $\partial D$ and $\partial D_1$ are of $C^{2}$, we can assume that $\partial{D}_{1}$ and $\partial{D}$ can be represented by 
$$x_{n}=\varepsilon+f(x'),\quad\mbox{and}\quad x_{n}=g(x'),
\quad
\textrm{for}\; |x'|\leq\,2R ,
$$
respectively, where $f$ and $g$ are two $C^{2}$ functions.   
Noting that $\partial D_1^*$ and $\partial D$ touch at the origin, and by the relative strict convexity assumption of $\partial D_1^*$ and $\partial D$, $f$ and $g$ are two $C^{2}$ functions satisfying
\begin{align}
f(x')&\geq g(x'), \quad\quad\mbox{for}\ 0\leq|x'|\leq 2R,\label{fg0}\\
f(0')&=g(0')=0,\quad \nabla_{x'}f(0')=\nabla_{x'}g(0')=0,\quad D^2(f-g)(0')>0.\label{fg1}
\end{align}
By \eqref{fg1} and Taylor's theorem, for small $R>0$, we further have
\begin{equation}\label{fg}
\kappa|x'|^2\leq f(x')-g(x')\leq\frac{1}{\kappa}|x'|^2, \quad \mbox{for}\,0\leq|x'|\leq 2R,\,\mbox{for some}\,\kappa>0.
\end{equation}

\begin{remark}
The smoothness assumptions of $\partial D$ and $\partial D_1$ can be relaxed. We can assume that $\partial D$ and $\partial D_1$ are of $C^{1,1}$ and $f$ and $g$ are two $C^{1,1}$ functions satisfying
\begin{equation}
\begin{split}
f(0')&=g(0')=0,\quad \nabla_{x'}f(0')=\nabla_{x'}g(0')=0,\\ \kappa|x'|^2&\leq f(x')-g(x')\leq\frac{1}{\kappa}|x'|^2,\quad\mbox{for}\ 0\leq|x'|\leq 2R .
\end{split}
\end{equation}
\end{remark}

Now we introduce some notations. For $0<r\leq\,2R $, we denote
$$B'_{r}(z'):=\left\{x'\in\mathbb{R}^{n-1}~\big|~|x'-z'|<r \right\}\subset\mathbb{R}^{n-1},$$
and
$$\Omega_r(z):=\left\{(x',x_{n})\in \mathbb{R}^{n}~\big|~g(x')<x_{n}<\varepsilon+f(x'),~x'\in\,B'_{r}(z')\right\},$$
with top and bottom boundaries
\begin{equation}
\begin{split}
 \Gamma_r^{+}(z):=&\left\{(x',x_{n})\in \mathbb{R}^{n}~\big|~x_n= \varepsilon+f(x'),~x'\in\,B'_{r}(z')\right\},\\
 \Gamma_r^{-}(z):=&\left\{(x',x_{n})\in \mathbb{R}^{n}~\big|~x_n= g(x'),~x'\in\,B'_{r}(z')\right\}.
\end{split}
\end{equation}
When there is no confusion, we also use notations $\Omega_r=\Omega_r(0)$, $\Gamma_r^{\pm}=\Gamma_r^{\pm}(0)$ and $B_r'=B_r'(0')$.
We further assume that the $C^2$ norms of $\partial{D}_{1}$ and $\partial{D}$ are bounded by a constant independent of $\varepsilon$. Throughout this paper, we call a constant is universal if it depends only on $n$, $\alpha$,     $R$, the upper bound of the $C^{2}$ norms of $\partial{D}_{1}$ and $\partial{D}$,  but not on $\epsilon$. For instance, $\kappa$, $\bar R$, $\tilde R$ and $C$ are such universal constants.

We denote 
$$\underset{\Omega_r(z)}\osc\,u:=\underset{\Omega_r(z)}\sup\,u-\underset{\Omega_r(z)}\inf\,u,$$
and introduce a $*$-norm:
\begin{equation}\label{*norm1}
\|\phi\|^{*}_{C^{\alpha}\big(\Gamma^{-}_{s}(z)\big)}:=\|\phi\|_{L^{\infty}\big(\Gamma^{-}_{s}(z)\big)}+\eta(z')^{\alpha}[\phi]_{C^{\alpha}\big(\Gamma^{-}_{s}(z)\big)}.
\end{equation}
For $ |z'|<2R$, set
\begin{equation}\label{deltaz}
\delta(z'):=\varepsilon+f(z')-g(z'),\quad \eta(z')=\varepsilon+|z'|^2.
\end{equation}
It is easy to verify that $\delta(z')\sim\eta(z')$. In this paper, $A\sim B$ means that there exists a universal constant $C$, independent of $z$ and $\epsilon$, such that $\frac{1}{C}A\leq B\leq CA$. 

Our main results of this paper are as follows.
 
\begin{prop}\label{prop_mainthm}
For $n\geq3$ and $0<\varepsilon\ll1$, let $u$ be a solution to \eqref{equinfty0}. Suppose that $D_1$ and $D$ satisfy \eqref{fg0}--\eqref{fg}. Then for some $\bar{R}<R$, there exists a universal constant $C$, such that
$$|\nabla u(z)|\leq C\eta(z')^{-1/2}\Big(\,\underset{\Omega_{\frac{1}{4}\sqrt{\eta(z')}}(z)}\osc\,u+\|\phi\|^{*}_{C^{\alpha}\big(\Gamma^{-}_{\frac{1}{4}\sqrt{\eta(z')}}(z)\big)}\Big),\quad\mbox{for}~z\in\Omega_{\bar{R}}.$$
\end{prop}

\begin{remark}
We can not conclude that $\eta(z')^{-1/2}$ is a pointwise upper bound for $\nabla u$ from Proposition
\ref{prop_mainthm}. This is because, as $\varepsilon\to0$, the uniform boundedness of $\underset{\Omega_{\frac{1}{4}\sqrt{\eta(z')}}(z)}\osc\,u$ is not  obvious. In Corollary \ref{cor11}, we will show that in a bigger region $\Omega_{\tilde R}$, $\underset{\Omega_{\tilde R} }\osc\,u$ may blow up as $\varepsilon\to0$.  To complete the proof for the upper bound estimate, the following Proposition is necessary.
\end{remark}

\begin{prop}\label{ubound}
For $n\geq3$ and $0<\varepsilon\ll1$, let $u$ be a solution to \eqref{equinfty0}. Suppose that $D_1$ and $D$ satisfy \eqref{fg0}--\eqref{fg}. Then there exists a constant $\tilde{R}<\bar{R}$, such that  
\begin{align}\label{uinfty}
\Big\|\,u-\fint_{\Omega\backslash\Omega_{\tilde{R}/2}}u\,\Big\|_{L^{\infty}({\Omega}\backslash\Omega_{3\tilde{R}/4})}\leq C\|\phi\|_{C^{\alpha}(\partial D)},
\end{align}
and 
\begin{align}\label{oscu}
\underset{\Omega_{\frac{1}{4}\sqrt{\eta(z')}}(z)}\osc\,u\leq C\|\phi\|_{C^{\alpha}(\partial D)},\quad\mbox{for}~z\in\Omega_{\tilde{R}}.
\end{align}
\end{prop}

By Proposition \ref{prop_mainthm}, Proposition \ref{ubound} and by applying the maximum principle to $\partial_n u$, we can immediately obtain the main theorem of this paper. 
\begin{theorem}\label{cor1}
For $n\geq3$ and $0<\varepsilon\ll1$, let $u$ be a solution to \eqref{equinfty0}. Suppose that $D_1$ and $D$ satisfy \eqref{fg0}--\eqref{fg}. Then for some $\tilde{R}$ ($\tilde{R}<\bar{R}$ as in Proposition \ref{ubound}), we have
\begin{equation}\label{upperbd}
|\nabla u(z)|\leq C\eta(z')^{-1/2}\|\phi\|_{C^{\alpha}(\partial D)},\quad\mbox{for}~z
\in\Omega_{\tilde{R}},
\end{equation}
and $
\|\nabla u\|_{L^{\infty}(\Omega\backslash\Omega_{\tilde{R}})}\leq C\|\phi\|_{C^{\alpha}(\partial D)}$. Moreover, 
\begin{equation}\label{upperbdun}
\|\partial_{n} u\|_{L^{\infty}(\Omega)}\leq C\|\phi\|_{C^{\alpha}(\partial D)}.
\end{equation}
\end{theorem}

We would like to point out that by slightly modifying our proof we can deal with the boundary estimates for the insulated problem with variable coefficients. We refer readers to \cite{DLY2}, the proof of Theorem 1.1 and discussion in Section 7  where they have shown the framework for reducing the variable coefficients case to the constant coefficients case \eqref{equinfty0}.

\begin{corollary}\label{cor11} 
Under the assumptions in Theorem \ref{cor1},

 (i) if $\phi(0')\neq0$, then $\underset{\Omega_{\tilde{R}}}\osc\,u$ may blow up at the rate of $|\ln\epsilon|$, that is, 
\begin{equation}\label{blowuposcu}
\underset{\Omega_{\tilde{R}}}\osc\,u\geq C|\phi(0')||\ln\epsilon|-C\|\phi\|_{C^{\alpha}(\partial D)},
\end{equation}
and
\begin{equation}\label{blowuposcu1}
\underset{\Omega_{\tilde{R}}}\osc\,u\leq C|\phi(0')||\ln\epsilon|+C\|\phi\|_{C^{\alpha}(\partial D)}.
\end{equation}

(ii) If $\phi(0')=0$, then the $\underset{\Omega_{\tilde R}}\osc\, u$ will not blow up. Furthermore, in this situation, the upper bound of $\nabla u$ can be improved. There exists a universal constant $\tilde \alpha\in(0,1)$ such that
$$|\nabla u(z)|\leq C\|\phi\|_{C^{\alpha}( \partial \Omega)} \delta(z')^{\frac{\tilde\alpha-1}{2}}
,\quad\mbox{for}~|z'|<\frac{1}{4}\bar{R}.$$
\end{corollary}

By Corollary \ref{cor11} (ii) , it is shown that the boundary estimates and interior estimates are consistent when $\phi(0')=0$. Next, we show that the blow-up order $\varepsilon^{-1/2}$ in Theorem \ref{cor1} is optimal in the following sense.
 
\begin{theorem}\label{lowerbound}
For $n\geq3$ and $0<\varepsilon\ll1$, let $u$ be a solution to \eqref{equinfty0}. Suppose that $D_1$ and $D$ are of $C^{2,\gamma}$ satisfying \eqref{fg0}--\eqref{fg}.
If $\phi(0')\neq0$, then there exists a fixed universal constant $C_0$ such that for sufficiently small $\varepsilon$ satisfying $$  |\ln \varepsilon| \varepsilon^{\frac{\tau}{2}} \tilde R^{-\tau} \leq \frac{ |\phi(0')|}{\,\,\,\,C_0\|\phi\|_{C^{\alpha}(\partial D)}},$$
where $\tau\in(0,1)$ is a universal constant, we have 
\begin{equation*}
\|\nabla u\|_{L^{\infty}(\Omega)}\geq \frac{1}{C_0\sqrt\epsilon}|\phi(0')|.
\end{equation*}
\end{theorem}

Finally, for the Dirichlet problem \eqref{equinfty-0}, we have the boundedness of $|\nabla u|$.

\begin{theorem}\label{thDir}
For $n\geq3$ and $0<\varepsilon\ll1$, let $u$ be a solution to \eqref{equinfty-0}. Suppose that $D_1$ and $D$ satisfy \eqref{fg0}--\eqref{fg}. Then
$$\|\nabla u\|_{L^{\infty}(\Omega)}\leq C\|\varphi\|_{C^{1,\alpha}(\partial D )}.$$
\end{theorem}
 
This paper is organized as follows. In Section 2,  we provide some preliminary calculations. In Section 3, we demonstrate the proof of  Proposition \ref{prop_mainthm} by establishing the local $L^2$ estimates for $\nabla u$, providing the $L^{\infty}$ estimates for the gradient of the solution to the related equations and using the bootstrap argument. In section 4, we make use of Proposition \ref{prop_mainthm} and some results on elliptic equations with degenerate coefficients to prove Proposition \ref{ubound}. Here we employ the Harnack inequality and the local estimates near the boundary  to obtain the global and local property of the solution to the degenerate elliptic equation. We prove Theorem \ref{cor1}, Corollary \ref{cor11} and Theorem \ref{lowerbound} in Section 5. Especially, the auxiliary function plays a very important role in the proof of the lower bound estimate of the gradient. Finally, the proof of Theorem \ref{thDir} is presented in Section 6.

\section{Prelimilaries}

In this section we prove some preliminary results. We focus on the following problem on the narrow region:
\begin{equation}\label{equinfty01}
\begin{cases}
\Delta{u}=0 & \mbox{in}~ \Omega_{2R},\\
\partial_{\nu}u=0&\mbox{on}~\Gamma_{2R}^{+},\\
\partial_{\nu}u=\phi(x)&\mbox{on}~\Gamma_{2R}^{-}.
\end{cases}
\end{equation}
 By means of a change of variables
\begin{equation}\label{changevariables} 
\begin{cases}
y'=x',\\
y_n=2\delta(z')\left(\frac{x_n-g(x')}{\delta(x')}-\frac{1}{2}\right),\\
\end{cases}
\quad\mbox{for}~(x',x_n)\in \Omega_{2R},
\end{equation}
where $\delta(z')$ is defined in \eqref{deltaz}, we transform the narrow region $\Omega_{2R}$ to a cylinder $\mathtt{Q}_{\delta,2R}$, with a height of  $2\delta(z')$, where $\mathtt{Q}_{\delta,s}$ is defined as follows:
$$\mathtt{Q}_{\delta,s}(z')=\mathtt{Q}_{\delta,s}(z):=\{y=(y',y_n)\in\R^n\big|~ y'\in\,B'_{s}(z'), |y_n|< \delta(z')\},$$
with corresponding top and bottom boundaries
$$\mathtt{\Gamma}_{\delta,s}^{\pm}(z')=\mathtt{\Gamma}_{\delta,s}^{\pm}(z):=\{y=(y',y_n)\in\R^{n}\big|~y'\in\,B'_{s}(z'),y_n=\pm\delta(z')\}.$$
For simplicity of notations, we also use $\mathtt{Q}_{\delta,s}:=\mathtt{Q}_{\delta,s}(0)$ and $\mathtt{\Gamma}_{\delta,s}^{\pm}:=\mathtt{\Gamma}_{\delta,s}^{\pm}(0).$

Setting $v(y)=u(x)$, then $v(y)$ satisfies
\begin{equation}\label{equ_v1}
\begin{cases}
\partial_i(a^{ij}(y)\partial_{j}v(y))=0\quad &\mbox{in}\ \mathtt{Q}_{\delta,2R} ,\\
a^{nj}(y)\partial_jv(y)=0\quad&\mbox{on}\ \mathtt{\Gamma}_{\delta,2R}^{+} ,\\
-a^{nj}(y)\partial_jv(y)=\psi(y')\quad&\mbox{on}\ \mathtt{\Gamma}_{\delta,2R}^{-} ,
\end{cases}
\end{equation}
where 
$$\psi(y')=\phi(y',g(y'))\sqrt{1+|\nabla_{y'}g(y')|^{2}},\quad~~~\mbox{for}~ |y'|<2R.$$ 
The coefficient matrix $(a^{ij}(y))$ in \eqref{equ_v1} is given by
\begin{align*}
(a^{ij}(y))=&\, \frac{(\partial_x y)(\partial_x y)^t}{\det(\partial_xy)}\\
=&\,\frac{\delta(y')}{2\delta(z')}
\left(
\begin{array}{ccccc}
1 &0 &\cdots &0 &  e^1(y)\\
0 &1 &\cdots &0 &  e^2(y)\\
\vdots & \vdots &\ddots &\vdots & \vdots\\
0 &0 &\cdots &1 &  e^{n-1}(y)\\
e^1(y)  & e^2(y) &\cdots  & e^{n-1}(y)  &\sum_{i=1}^{n} e^i(y)^2
\end{array}
\right),
\end{align*}
where 
\begin{align}\label{def_ei}
e^i(y)=\partial_{y_i}g\frac{y_n-\delta(z')}{\delta(y')}-\partial_{y_i}f\frac{y_n+\delta(z')}{\delta(y')},~ 1\leq i\leq n-1,~
\mbox{and}~
e^n(y)=\frac{2\delta(z')}{\delta(y')}.
\end{align}
In the sequel, $\lambda$, $\Lambda$, $R$, $\bar{R}$, $\tilde{R}$, $C_{0}$ are some fixed universal constants, while the universal constant $C$ may vary from line to line. 

First, we have the following lemma.
\begin{lemma}\label{lem21}
There exist universal constants $\lambda,\Lambda$, and $C$, such that for $|z'|<R$ and for $y\in\mathtt{Q}_{\delta,\frac{1}{4}\sqrt{\eta(z')}}(z)$, the following estimates hold:
\begin{align}\label{ryz}
\frac{1}{8}\eta(y')\leq\eta(z')\leq8\eta(z')
\end{align}
\begin{align}\label{aii}
\lambda<|a^{ii}(y')|<\Lambda,\quad 1\leq i\leq n;
\end{align}
\begin{align}\label{ani1}
|a^{nj}(y)|=|a^{jn}(y)|\leq C\eta(z')^{1/2}, \quad 1\leq j\leq n-1;
\end{align}
and for $ 0<\mu<1$, 
\begin{align}\label{aii2}
 [a^{ii}]_{C^{\mu}(\mathtt{Q}_{\delta,\frac{1}{4}\sqrt{\eta(z')}}(z'))} \leq C\eta(z')^{-\mu/2},\quad 1\leq i\leq n-1,
\end{align}
and
\begin{align}\label{ani2}
 [a^{nj}]_{C^{\mu}(\mathtt{Q}_{\delta,\frac{1}{4}\sqrt{\eta(z'))}}(z)} \leq C\eta(z')^{\frac{1-2\mu}{2}}, \quad1\leq j\leq n-1.
\end{align}
\end{lemma}

\begin{proof}
Without loss of generality, we assume that for $|z'|<R$, there holds 
\begin{align}\label{fg_R}
R\,\Big(\|f\|_{C^{2 }(B'_{R}(0))}+\|g\|_{C^{2 }(B'_{R}(0))}\Big)<1/4.
\end{align}
Hence this $R$ depends only on the upper bound of the $C^{2}$ norms of $\partial{D}_{1}$ and $\partial{D}$. By a direct calculation, for $|z'|<R$ and $y\in\mathtt{Q}_{\delta,\frac{1}{4}\sqrt{\eta(z')}}(z)$,
$$\eta(y')\leq\epsilon+2(|z'|^{2}+|y'-z'|^{2})\leq(2+\frac{1}{4})\eta(z')\leq8\eta(z').$$
On the other hand, for the lower bound, if $|z'|\geq2\sqrt{\epsilon}$, then $|y'|\geq|z'|-|y'-z'|\geq|z'|-\frac{1}{2}|z'|\geq\frac{1}{2}|z'|$, thus
$$\eta(y')\geq\epsilon+\frac{1}{4}|y'|^{2}\geq\frac{1}{8}\eta(z'),$$
while if $|z'|\leq2\sqrt{\epsilon}$, then $\frac{\eta(y')}{\eta(z')}\geq\frac{\epsilon}{8\epsilon}\geq\frac{1}{8}$. Therefore, for $|z'|<R$ and $y\in\mathtt{Q}_{\delta,\frac{1}{4}\sqrt{\eta(z')}}(z)$, we have
\begin{align}\label{ryz1}
 \frac{1}{8}\leq\frac{\eta(z')}{\eta(y')}\leq8.
\end{align}
Combining the fact that $\delta(x')\sim\eta(x')$ and \eqref{ryz1}, we have 
\begin{equation}
\lambda<a^{ii}(y')=\frac{\delta(y')}{2\delta(z')}<\Lambda, \quad \mbox{for} 1\leq i\leq n-1.
\end{equation}

Because $|\nabla f(0')|=|\nabla g(0')|=0$, then using the mean value theorem yields
$$
|\nabla f(y')|\leq \|\nabla^2f\|_{L^{\infty}(B'_{2R}(0))}|y'|\leq C|y'|,\quad \mbox{and}\quad |\nabla g(y')|\leq C|y'|.$$
Hence, for $y\in\mathtt{Q}_{\delta,\frac{1}{4}\sqrt{\eta(z')}}(z)$,
\begin{equation}\label{eiy}
|e_i(y)|\leq C(|\nabla g(y)|+|\nabla f(y)|)\leq\,C|y'|\leq\,C(|z'|+|y'-z'|)\leq\,C\eta(z')^{1/2},
\end{equation}
and thus, $\sum_{i=1}^{n-1}(e_i)^2\leq C\eta(z')$. Therefore, using \eqref{ryz1} again,
$$\lambda<\,e_{n}(y)\leq\,a^{nn}=\frac{\delta(y')}{2\delta(z')}\sum_{i=1}^{n-1}(e_i)^2+\frac{2\delta(z')}{\delta(y')}<\Lambda.$$
While, for $1\leq\,j\leq n-1$, by virtue of \eqref{eiy},
$$|a^{nj}(y)|=\frac{\delta(y')}{2\delta(z')}|e_j(y)|\leq\,C\eta(z')^{1/2}.$$
So, \eqref{aii} and \eqref{ani1} hold.

For $|z'|<R$ and $y\in\mathtt{Q}_{\delta,\frac{1}{4}\sqrt{\eta(z')}}(z)$, since 
\begin{equation}\label{ndeltay}
|\partial_{y_{i}}\delta(y)|\leq\,\|\nabla^{2}(f-g)\|_{L^{\infty}(B'_{R}(0))}|y'|\leq\,C|y'|\leq\,C\eta(z')^{1/2},
\end{equation}
it follows that, for $1\leq\,i\leq n-1$, 
$$[a^{ii}]_{C^{\mu}(Q_{\delta , \frac{1}{4}\sqrt{\eta(z')} }(z))}=\frac{1}{2\delta(z')}[\delta(y)]_{C^{\mu}(Q_{\delta , \frac{1}{4}\sqrt{\eta(z')} }(z))}\leq C\eta(z')^{-\frac{\mu}{2}}.$$
That is, \eqref{aii2} holds true.

For $a^{nj}(y),~ 1\leq j\leq n-1$, since
$$a^{nj}(y)=\partial_{y_j}g\frac{y_n-\delta(z')}{2\delta(z')}-\partial_{y_j}f\frac{y_n+\delta(z')}{2\delta(z')},$$
it is direct to check that
$$|\nabla_{y'}a^{nj}(y)|\leq\,C,\quad |\partial_{y_n}a^{nj}(y)|\leq C\eta(z')^{-\frac{1}{2}}$$
So that, for $\bar y, \tilde y\in Q_{\delta , \frac{1}{4}\sqrt{\eta(z')} }(z)$, we have $|\tilde y-\bar y|\leq C\eta(z')^{1/2}$, and
\begin{equation*}
\begin{aligned}
 \frac{|a^{nj}(\bar y)-a^{nj}(\tilde y)|}{|\bar y-\tilde y|^{\mu}} 
\leq&\, \frac{|a^{nj}(\bar y',\bar y_n)-a^{nj}(\bar y',\tilde y_n)|}{|\bar y-\tilde y|^{\mu}} +\frac{|a^{nj}(\bar y',\tilde y_n)-a^{nj}(\tilde y',\tilde y_n)|}{|\bar y-\tilde y|^{\mu}}\\
\leq&\, |\partial_{y_{n}}a^{nj}||\bar y_n-\tilde y_n|^{1-\mu}+|\nabla_{y'}a^{nj}||\bar y'-\tilde y'|^{1-\mu}\\
\leq&\, C\eta(z')^{-1/2}\eta(z')^{1-\mu}+C\eta(z')^{\frac{1-\mu}{2}}\leq C\eta(z')^{\frac{1-2\mu}{2}}.
\end{aligned}
\end{equation*}
Thus, \eqref{ani2} is proved. The proof of Lemma \ref{lem21} is completed.
\end{proof}

We now define the norms:
\begin{align*}
\|\nabla v\|_{C^{\mu}(\mathtt{Q}_{\delta,s}(z'))}^*&:=\|\nabla v\|_{L^{\infty}(\mathtt{Q}_{\delta,s}(z'))}+\eta(z')^{\mu}[\nabla v]_{C^{\mu}(\mathtt{Q}_{\delta,s}(z'))},\end{align*}
and
\begin{align*}
\|\nabla u\|_{C^{\mu}(\Omega_{r}(z))}^*&:=\|\nabla u\|_{L^{\infty}(\Omega_{r}(z))}+\eta(z')^{\mu}[\nabla u]_{C^{\mu}(\Omega_{r}(z))}.
\end{align*}
Then
\begin{lemma}\label{uvsqrtdelta}
 For $|z'|<R$ and for $0\leq s\leq \frac{1}{4}\sqrt{\eta(z')}$,  
\begin{align}
\|\nabla v\|_{C^{\mu}(\mathtt{Q}_{\delta, s}(z))}^*\sim \|\nabla u\|_{C^{\mu}(\Omega_{s}(z))}^*.
\end{align}
 \end{lemma}

\begin{proof}
Under transform \eqref{changevariables}, we write $\bar y=y(\bar x)$ and $\tilde y=y(\tilde x)$. We claim that 
\begin{equation}\label{rxy}
|\tilde x-\bar x|\sim |\tilde y- \bar y|,\quad\mbox{for}~\tilde x,\bar x\in\Omega_{s}(z),~s\leq \frac{1}{4}\sqrt{\eta(z')}.
\end{equation}
Indeed, it suffices to prove that
\begin{equation}\label{rxy1}
|\tilde x'-\bar x'|+|\tilde x_n-\bar x_n|\sim |\tilde y'-\bar y'|+|\tilde y_n-\bar y_n|.
\end{equation}
It is clear that $|\tilde x'-\bar x'|=|\tilde y'-\bar y'|$. We only need to estimate $|\tilde x_n-\bar x_n|$. For this, by definitions,
\begin{equation}\label{xn}
\begin{aligned}
\bar x_n-\tilde x_n 
=&\,\Big((\frac{\bar y_n}{2\delta(z')}+\frac{1}{2})\delta(\bar y')+g(\bar y')-\frac{\epsilon}{2}\Big)-\Big((\frac{\tilde y_n}{2\delta(z')}+\frac{1}{2})\delta(\tilde y')+g(\tilde y')-\frac{\epsilon}{2}\Big) \\
=&\,\frac{\bar y_n}{2\delta(z')}\delta(\bar y')-\frac{\tilde y_n}{2\delta(z')}\delta(\tilde y')+\frac{1}{2}(f(\bar y')+g(\bar y'))-\frac{1}{2}(f(\tilde y')+g(\tilde y')) \\
=&\,\frac{\delta(\bar y')}{2\delta(z')}(\bar y_n-\tilde y_n)+\frac{\tilde y_n}{2\delta(z')}(\delta(\bar y')-\delta(\tilde y'))+\frac{1}{2}((f+g)(\bar y')-(f+g)(\tilde y')).
\end{aligned}
\end{equation}
By means of the assumption \eqref{fg_R}, we have
$$|f(\bar y')-f(\tilde y')|\leq|\nabla f(\theta\bar y'+(1-\theta)\tilde y')||\bar y'-\tilde y'|\leq\|\nabla^{2} f\|_{L^{\infty}(B'_{R}(0))}R_{1}|\bar y'-\tilde y'|\leq\frac{1}{4}|\bar y'-\tilde y'|,$$
$$|g(\bar y')-g(\tilde y')|\leq\frac{1}{4}|\bar y'-\tilde y'|,\quad\mbox{and}~
|\delta(\bar y')-\delta(\tilde y')|\leq\,\frac{1}{4}|\bar y'-\tilde y'|.$$
This, together with \eqref{ryz} and \eqref{xn}, yields
\begin{equation*}
\begin{aligned}
|\bar x_n-\tilde x_n|\leq&\, C|\bar y_n-\tilde y_n|+\frac{1}{2}|\bar y'-\tilde y'|,\quad\mbox{and}~
|\bar x_n-\tilde x_n|\geq C|\bar y_n-\tilde y_n|-C|\bar y'-\tilde y'|.
\end{aligned}
\end{equation*}
Thus, \eqref{rxy1} is proved.

By the chain rule, $\partial_{y_{i}}v(y)=\partial_{x_{i}}u(x)+\partial_{x_{n}}u(x)\partial_{y_{i}}x_{n}$, we have
\begin{align*}
\frac{|\partial_{y_{i}}v(\bar{y})-\partial_{y_{i}}v(\tilde{y})|}{|\bar{y}-\tilde{y}|^{\mu}}
\leq&\,\frac{|\partial_{x_{i}}u(\bar{x})-\partial_{x_{i}}u(\tilde{x})|}{|\bar{y}-\tilde{y}|^{\mu}}
+\frac{|\partial_{x_{n}}u(\bar{x})\partial_{y_{i}}x_{n}(\bar{y})-\partial_{x_{n}}u(\tilde{x})\partial_{y_{i}}x_{n}(\tilde{y})|}{|\bar{y}-\tilde{y}|^{\mu}}.
\end{align*}
To estimate it, we need to calculate $\nabla^{2}_{y}\,x$. For $1\leq i\leq n-1$, $1\leq i\leq n-1$, 
\begin{equation*}
\begin{aligned}
\partial_{y_i}x_{n}=&\,\Big(\frac{y_n}{2\delta(z')}+\frac{1}{2}\Big)\partial_{y_i}\delta(y')+\partial_{y_i}g(y'),\\
\partial^2_{y_{i}y_{j}}x_n=&\,\Big(\frac{y_n}{2\delta(z')}+\frac{1}{2}\Big)\partial^2_{y_{i}y_{j}}\delta(y')+\partial^2_{y_{i}y_{j}} g(y'),\\
\partial_{y_{n}} x_n=&\,\frac{\delta(y')}{2\delta(z')},\quad\mbox{and}~
\partial^2_{y_{n}y_{j}} x_n=\,\frac{1}{2\delta(z')}\partial_{y_{j}}\delta(y').
\end{aligned}
\end{equation*}
Because $|\nabla f(y')|+|\nabla g(y')|\leq C|y'|\leq C\sqrt{\eta(y')}\leq C\sqrt{\eta(z')}$, we have   
\begin{equation}\label{xny1}
|\nabla x_n(y)|\leq\,C,\quad |\nabla^2 x_n(y)|\leq C\eta(z')^{-1/2}.
\end{equation}
This implies 
\begin{equation}\label{xny2}
[\nabla x_n(y)]_{C^{\mu}(\mathtt{Q}_{\delta,s}(z))}\leq \|\nabla^2 x_n(y)\|_{L^{\infty}(\mathtt{Q}_{\delta,s}(z))}\sqrt{\eta(z')}^{1-\mu}.
\end{equation}
Thus, by using \eqref{rxy}, \eqref{xny1}, \eqref{xny2} and the chain rule, we have
 $$\|\nabla v\|_{C^{\mu}(\mathtt{Q}_{\delta, s}(z))}^*\leq C \|\nabla u\|_{C^{\mu}(\Omega_{s}(z))}^*.$$ 
The other direction is similar. Thus, Lemma \ref{uvsqrtdelta} is proved. 
\end{proof}
Similar as \eqref{*norm1}, we define
$$\|\psi\|_{C^{\alpha}(\mathtt{\Gamma}^{-}_{\delta, s}(z))}^*:=\|\psi\|_{L^{\infty}(\mathtt{\Gamma}^{-}_{\delta, s}(z))}+\eta(z')^{\alpha}[\psi]_{C^{\alpha}(\mathtt{\Gamma}^{-}_{\delta, s}(z))}.$$
By a direct calculation, it is easy to see  that 
\begin{equation}
\|\phi\|_{C^{\alpha}(\Gamma^{-}_{s}(z))}^*\sim \|\psi\|_{C^{\alpha}(\mathtt{\Gamma}^{-}_{\delta, s}(z))}^*, \quad\mbox{for}~ s>0.
\end{equation}

\section{proof of Proposition \ref{prop_mainthm}}\label{section3}

In this section we are dedicated to proving Proposition \ref{prop_mainthm}. As mentioned in Section 1, the proofs of interior estimates for insulated conductivity problem do not directly apply to the boundary case. Our strategies in this section are as follows.  Firstly, we establish the estimate $\nabla (u-\bar u)$ in the $L^2$ sense by using an iteration technique, where $\bar u$ is defined as \eqref{barequ_v11} below. Then by utilizing the $L^{\infty}$ estimate for $\nabla\bar u$ and employing a bootstrap argument, we improve the pointwise upper bound for $\nabla u$ from $\eta(x')^{-1}$ to $\eta(x')^{-1/2}$.

\subsection{$C^{\mu}$ estimates of $\nabla v$}

\begin{lemma}\label{lem_estv}
Let $v$ be the solution to \eqref{equ_v1}. Then for $|z'|<R$, and for some $\mu\in(0,\frac{1}{2}]$, we have
\begin{equation}\label{vdeltabarv}
\|\nabla v\|_{C^{\mu}(\mathtt{Q}_{\delta,\eta(z')/16}(z))} ^*\leq  C\eta(z')^{-n/2}\|\nabla v\|_{L^2(\mathtt{Q}_{\delta,\eta(z')/8}(z))}+ C\|\psi\|_{C^{\alpha}(\mathtt{\Gamma}^{-}_{\delta, \eta(z')/8}(z))}^*,
\end{equation}
and  
 \begin{equation}\label{vdelta}
\|\nabla v\|_{C^{\mu}(\mathtt{Q}_{\delta,\eta(z')/16}(z))} ^*\leq C\eta(z')^{-1}\underset{\mathtt{Q}_{\delta,\eta(z')/8}(z)}\osc~ v +C\|\psi\|_{C^{\alpha}(\mathtt{\Gamma}^{-}_{\delta,  \eta(z')/8}(z))}^*.
\end{equation}

\end{lemma}
\begin{proof}
We first rescale domain $\mathtt{Q}_{\delta,\eta(z')/8}(z)$ to a unit size. Set
\begin{align*}
\hat{a}^{ij}(w):=&a^{ij}(z'+ \frac{\eta(z')}{8}w',\delta(z')w_n),\quad 1\leq\,i,j\leq n-1,\\
 \hat{a}^{in}(w):=&\frac{\eta(z')}{8\delta(z')}a^{in}(z'+ \frac{\eta(z')}{8}w',\delta(z')w_n),\\
 \hat{a}^{nn}(w):=&\frac{\eta(z')^2}{64\delta(z')^2}a^{nn}(z'+ \frac{\eta(z')}{8}w',\delta(z')w_n),
\end{align*}
and
$$\hat{v}(w):=v(z'+\frac{\eta(z')}{8}w',\delta(z')w_n),\quad\hat{\psi}(w'):=\psi(z'+\frac{\eta(z')}{8}w'),$$
for $w=(w',w_n)\in Q_{1,1}$, where
$$Q_{1,s}=Q_{1,s}(0):=\{y=(y',y_n)\in\R^n,\big| y'\in\,B'_{s}(0'), |y_n|< 1\},\quad\mbox{for}~s>0.$$ 
Then $\hat{v}(w)$ satisfies
\begin{equation}\label{equ_w}
\begin{cases}
\partial_i(\hat{a}^{ij}(w)\partial_{j}\hat{v}(w))=0\quad &\mbox{in}\ Q_{1,1},\\
\quad a_{\delta}^{nj}(w)\partial_{j}\hat{v}(w)=0\quad&\mbox{on}\ \Gamma_{1,1}^{+}, \\
~~-a_{\delta}^{nj}(w)\partial_{j}\hat{v}(w)= \frac{\eta(z')^2}{64\delta(z')}\hat{\psi}(w)\quad&\mbox{on}\ \Gamma_{1,1}^{-},\\
\end{cases}
\end{equation}
where $\Gamma_{1,1}^{\pm}:=\{y=(y',y_n)\in\R^n,\big| y'\in\,B'_{s}(0'), y_n=\pm 1\}.$ 

Define the norm, as in \cite{L1},
$$\|\psi\|_{\alpha,p;\partial{D}}:=\|\psi\|_{L^{p}(\partial{D})}+\langle \psi\rangle_{\alpha,p;\partial D},\quad\alpha\in(0,1),$$
and
$$\langle \psi\rangle_{\alpha,p;\partial D}:=\Big(\int_{\partial D}\int_{\partial D}\frac{|\psi(x)-\psi(y)|^{p}}{|x-y|^{n-2+p\alpha}}dS_{x}dS_{y}\Big)^{1/p}.$$
Then, for $p>0$, 
$$\|\hat{\psi}\|_{L^p(\Gamma_{1,1}^{-})}\leq C\|\hat{\psi}\|_{L^{\infty}(\Gamma_{1,1}^{-})} \leq C\|\psi\|_{L^{\infty}(\mathtt{\Gamma}^{-}_{\delta, \frac{1}{8}\eta(z')}(z))}\leq C\|\psi\|_{L^{\infty}(\Gamma^{-}_{ \frac{1}{8}\eta(z')}(z'))},$$
and for $p<\frac{2}{1-\alpha}$,
\begin{align*}
\langle\hat{\psi}\rangle_{1-\frac{1}{p},p;\Gamma^{-}_{1,1}}=&\,\Big(\int_{\Gamma^{-}_{1,1}}\int_{\Gamma^{-}_{1,1}}\frac{|\hat{\psi}(w_1)-\hat{\psi}(w_2)|^{p}}{|w_1-w_2|^{n-3+p}}dS_{w_1}dS_{w_2}\Big)^{1/p} \\
\leq&\,[\hat{\psi}]_{C^{\alpha}(\Gamma_{1,1}^{-})}\Big(\int_{|w_2'|\leq1}\int_{|w_1'|\leq1}|w_1'-w_2'|^{-n+3+(\alpha-1)p}dS_{w_1'}dS_{w_2'}\Big)^{1/p} \\
\leq&\,C[\hat{\psi}]_{C^{\alpha}(\Gamma_{1,1}^{-})}\leq\, C\eta(z')^{\alpha}[\psi]_{C^{\alpha}(\Gamma^{-}_{\eta(z')}(z'))}.
\end{align*}
Thus, for $p<\frac{2}{1-\alpha}$,
\begin{equation}\label{phiCalpha}
\|\hat{\psi}\|_{\alpha,p;\Gamma_{1,1}^{-}}
\leq\,C\|\hat{\psi}\|_{C^{\alpha}(\Gamma_{1,1}^{-})}\leq\,C \|\psi\|_{C^{\alpha}(\Gamma^{-}_{\delta(z')}(z))}^*.
\end{equation}

By using the local boundary $W^{2,p}$ estimates for domains with a $C^{1,\gamma}$ boundary portion, for elliptic equations with Neumann boundary condition, (see, e.g. Theorem 6.27 in \cite{L1}), we have, for any given constant $a$, and together with \eqref{phiCalpha},
\begin{align}\label{bootstrap0}
\| \hat{v}-a\|_{W^{2,p_{i}}(Q_{1,\frac{1}{2}+\frac{i}{2k+2}})}\leq&\, C\left(\| \hat{v}-a\|_{L^{p_{i}}(Q_{1,\frac{1}{2}+\frac{i+1}{2k+2}})}+\eta(z')\|\hat{\psi}(w)\|_{1-\frac{1}{p_{i}},p_{i}; \Gamma_{1,1}^{-}}\right)\nonumber\\
\leq&\,C\left(\| \hat{v}-a\|_{L^{p_{i}}(Q_{1,\frac{1}{2}+\frac{i+1}{2k+2}})}+\eta(z')\|\hat{\psi}\|_{C^{\alpha}(\Gamma_{1,1}^{-})}\right),
\end{align}
where $p_{i}<\frac{2}{1-\alpha}$, $i=0,1,2,\dots, k(n)$, and $k(n)$ is some finite integer depending only on $n$. For $n=3$, we take $k=1$ and $p_{1}=2<\frac{2}{1-\alpha}$ in \eqref{bootstrap0}, and obtain
\begin{equation}\label{bootstrap1}
\begin{split}
 \| \hat{v}-a_{1}\|_{W^{2,2}(Q_{1,\frac{3}{4}})}\leq C\left(\| \hat{v}-a_{1}\|_{L^{2}(Q_{1,1})}+\eta(z')\|\hat{\psi}\|_{C^{\alpha}(\Gamma_{1,1}^{-})}\right).
\end{split}
\end{equation}
Since for $n=3$, by the embedding theorem, $W^{1,2}\hookrightarrow L^{6}$, we choose $n<p_{0}<\frac{2}{1-\alpha}$, such that
\begin{equation}\label{bootstrap11}
\| \nabla(\hat{v}-a_{1})\|_{L^{p_{0}}(Q_{1,\frac{3}{4}})}\leq \| \hat{v}-a_{1}\|_{W^{2,2}(Q_{1,\frac{3}{4}})}.
\end{equation}
For this $p_{0}$, we apply \eqref{bootstrap0} again and use the embedding theorem $W^{2,p_{0}}\hookrightarrow C^{1,\mu}$ with $\mu= \min\{1-\frac{n}{p_{0}},\frac{1}{2}\}$, to derive
\begin{equation}\label{bootstrap2}
\begin{split}
&\|\nabla  \hat{v}\|_{C^{0,\mu}(Q_{1,\frac{1}{2}})}=\,\|\nabla ( \hat{v}-a_{2})\|_{C^{0,\mu}(Q_{1,\frac{1}{2}})}\\
\leq&\| \hat{v}-a_{2}\|_{W^{2,p_{0}}(Q_{1,\frac{1}{2}})}\leq C\left(\| \hat{v}-a_{2}\|_{L^{p_{0}}(Q_{1,\frac{3}{4}})}+\eta(z')\|\hat{\psi}\|_{C^{\alpha}(\Gamma_{1,1}^{-})}\right).
\end{split}
\end{equation}
Let $( \hat{v})_{Q_{1,s}}=\frac{1}{|Q_{1,s}|}\int_{Q_{1,s}} \hat{v}$ be the average of $\hat{v}$ over $Q_{1,s}$. Taking $a_{2}=(\hat{v})_{Q_{1,\frac{3}{4}}}$ in \eqref{bootstrap2},  by virtue of the Poincar\'{e} inequality, 
\begin{align}\label{boot11}
\|\hat{v}-(\hat{v})_{Q_{1,\frac{3}{4}}}\|_{L^{p_0}(Q_{1,\frac{3}{4}})}\leq&\, C\|\nabla(\hat{v}-(\hat{v})_{Q_{1,\frac{3}{4}}})\|_{L^{p_0}(Q_{1,\frac{3}{4}})}
=\,  C\|\nabla(\hat{v}-(\hat{v})_{Q_{1,1}})\|_{L^{p_0}(Q_{1,\frac{3}{4}})}.
\end{align}
Now taking $a_{1}=(\hat{v})_{Q_{1,1}}$ in \eqref{bootstrap1} and \eqref{bootstrap11}, combining with \eqref{bootstrap1}--\eqref{bootstrap2} leads to
\begin{align}\label{boot1}
\|\nabla  \hat{v}\|_{C^{0,\mu}(Q_{1,\frac{1}{2}})}\leq&\, C\left(\| \hat{v}-(\hat{v})_{Q_{1,1}}\|_{L^{2}(Q_{1,1})}+\eta(z')\|\hat{\psi}\|_{C^{\alpha}(\Gamma_{1,1}^{-})}\right).
\end{align}
For higher dimensions $n\geq4$, we choose $\frac{1}{p_{i+1}}=\frac{1}{p_{i}}+\frac{2}{n}$, $i=0,1,\dots,k(n)$, where $k(n)$ is the smallest integer such that $p_{k(n)}\geq2$. By using \eqref{bootstrap0} and a bootstrap argument, we finally have \eqref{boot1} holds. Then using the Poincar\'e inequality again and rescaling back to $v$, we can obtain \eqref{vdeltabarv}.
\end{proof}

\begin{remark}\label{u-1}
By Lemma \ref{uvsqrtdelta}, \eqref{ryz} and \eqref{vdelta}, we have
\begin{equation}\label{lem1-1}
\|\nabla u\|^{*}_{C^{\mu}(\mathtt{Q}_{\delta,\frac{1}{8}\sqrt{\eta(z')}}(z))}\leq  C\eta(z')^{-1}\underset{\Omega_{ {\frac{1}{4}\sqrt{\eta(z')}}}(z)}\osc~ u+C \|\phi\|^*_{C^{\alpha}(\Gamma^{-}_{\frac{1}{4}\sqrt{\eta(z')}}(z))}. 
\end{equation}
\end{remark}

We rewrite the equation in \eqref{equ_v1} as 
\begin{equation}\label{v1}
\sum_{i=1}^{n-1}\partial_i\Big(a^{ii}(y)\partial_{i}v(y)\Big)+\sum_{i=1}^{n-1}\partial_iF^i+\sum_{j=1}^{n}\partial_n\Big(a^{nj}(y)\partial_{j}v(y)\Big)=0\quad \mbox{in}\ \mathtt{Q}_{\delta,R}(z),
\end{equation}
where $a^{ii}=\frac{1}{2}\frac{\delta(y')}{\delta(z')}$ and $F^i=a^{in}\partial_nv$. Set
\begin{equation}\label{barequ_v1}
\bar v(y'):=\fint_{-\delta(z')}^{\delta(z')}v(y',y_n)dy_n,\quad|y'|<R.
\end{equation}
In fact, 
\begin{equation}\label{barequ_v11}
\bar v(y')=\fint_{-\delta(z')}^{\delta(z')}v(y',y_n)dy_n=\fint_{g(x')}^{\epsilon+f(x')}u(x',x_n)dx_n:=\bar u(x').
\end{equation}

Using the boundary condition in \eqref{equ_v1}, 
$$\sum_{j=1}^{n}a^{nj}\partial_jv|_{\mathtt{\Gamma}_{\delta,R}^{+}(z)}-a^{nj}\partial_jv|_{\mathtt{\Gamma}_{\delta,R}^{-}(z)}=\psi(y'),$$
it follows that
$$\sum_{j=1}^{n}\fint_{-\delta(z')}^{\delta(z')}\partial_n (a^{nj}\partial_jv)dy_n=\,\frac{\psi(y')}{2\delta(z')}.$$
Hence, taking average with respect to $y_{n}$ to equation \eqref{v1}, we have $\bar v(y')$ satisfies
\begin{align}\label{barv}
\sum_{i=1}^{n-1}\partial_i(a^{ii}\partial_i\bar v)+ \sum_{i=1}^{n-1}\partial_i \bar F^i
=\,\frac{\psi(y')}{2\delta(z')},\quad\mbox{where}~\bar F^i=\fint_{-\delta(z')}^{\delta(z')}a^{in}\partial_nvdy_n.
\end{align}
Because $\partial_n\bar v=0$, then \eqref{barv} can also be rewritten as
\begin{equation}\label{barv1}
\sum_{i=1}^{n}\partial_i\big(a^{ii}\partial_i\bar v\big)+ \sum_{i=1}^{n-1}\partial_i\bar F^i=\frac{\psi(y')}{2\delta(z')}.
\end{equation}

In \eqref{vdeltabarv}, since
$$\|\nabla v\|_{L^2(\mathtt{Q}_{\delta,\eta(z')/8}(z))}\leq\,\|\nabla (v-\bar v)\|_{L^2(\mathtt{Q}_{\delta,\eta(z')/8}(z))}+\|\nabla \bar v\|_{L^2(\mathtt{Q}_{\delta,\eta(z')/8}(z))},$$
we next estimate the $L^{2}$ norms of $\nabla (v-\bar v)$ and $\nabla \bar v$ in the following two lemmas.

\subsection{$L^{2}$ estimates of $\nabla(v-\bar{v})$}
We substracte \eqref{v1} from \eqref{barv1} to obtain
\begin{equation}\label{vminusbarv}
-\sum_{i=1}^{n}\partial_i\big(a^{ii}\partial_i (v-\bar v)\big)=\sum_{i=1}^{n-1}\partial_i\big(F^i- \bar F^i\big)+
\sum_{j=1}^{n-1}\partial_n(a^{nj}\partial_jv)+\frac{\psi(y')}{2\delta(z')},
\end{equation}
where $a^{ii}=\frac{\delta(y')}{2\delta(z')}$ and $a^{nn}=\frac{\delta(y')}{2\delta(z')}\sum_{i=1}^{n}(e_i)^2$.

To estimate $\nabla(v-\bar v)$, we first estimate $|F^i-\bar F^i|$. Since $\partial_n\bar v=0$, then
$$
F^i-\bar F^i=a^{in}\partial_n(v-\bar v)+\fint_{-\delta(z')}^{\delta(z')}a^{in}\partial_n(\bar{v}-v)dy_n.
$$
By using \eqref{ani1}, 
$$|a^{in}\partial_n(v-\bar v)|\leq C\delta(z')^{1/2}|\nabla(v-\bar v)|,$$
and furthermore,
\begin{align*} 
\Big|\fint_{-\delta(z')}^{\delta(z')}a^{in}\partial_n(\bar{v}-v)dy_n\Big|^2
\leq&\,\frac{1}{4\delta(z')^2}\Big(\int_{-\delta(z')}^{\delta(z')}|a^{in}\partial_n(\bar{v}-v)|dy_n\Big)^2\\
\leq&\,\frac{1}{2\delta(z')} \int_{-\delta(z')}^{\delta(z')}|a^{in}\partial_n(\bar{v}-v)|^2dy_n 
\leq\, C\int_{-\delta(z')}^{\delta(z')}|\nabla(v-\bar v)|^2dy_n.
\end{align*}
Thus, 
 \begin{align}\label{F-barF}
|F^i-\bar F^i|\leq C\delta(z')^{1/2}|\nabla(v-\bar v)|+C\Big(\int_{-\delta(z')}^{\delta(z')}|\nabla(v-\bar v)|^2dy_n \Big)^{1/2}.
 \end{align}

For the $L^{2}$ estimate of $|\nabla v-\nabla\bar v|$, we have

\begin{lemma}\label{lem_v-barv}
Let $v$ and $\bar v$ be the corresponding solutions to \eqref{equ_v1} and \eqref{barv1}.
Then there exists $\bar R$ $(<R)$, such that for $|z'|<\bar R$, 
  \begin{align}\label{v-barv}
 \|\nabla(v-\bar v)\|_{L^2({\mathtt{Q}_{\delta,\eta(z')/8}(z')})}
 \leq&\, C\eta(z')^{n/2}\Big(\|\nabla v\|_{L^{2}(\mathtt{Q}_{\delta, \sqrt{\eta(z')}/64}(z))}\nonumber\\&+\eta(z')^{1/2}\|\nabla\bar v\|_{L^{\infty}(B'_{ \sqrt{\eta(z')}/64}(z'))}+\|\psi\|_{L^{\infty}(\mathtt{\Gamma}^{-}_{\delta,  \sqrt{\eta(z')}/64}(z))}\Big).
 \end{align}
\end{lemma}
\begin{proof}
For $0<t<s<\frac{\sqrt{\eta(z')} }{8} $, $|z'|<R$, let $\xi$ be a cutoff function satisfying $0\leq\xi(x')\leq1$, 
$$\xi(y')=1, ~\mbox{if}~ |y'-z'|<t,\quad\xi(y')=0, ~\mbox{if}~|y'-z'|>s,$$ and $|\nabla_{x'}\xi(x')|\leq\frac{2}{s-t}$. Multiplying the equation in \eqref{vminusbarv} by $ (v-\bar v)\xi^{2}$, and integrating by parts, yields
\begin{equation}\label{integratebyparts}
\begin{split}
&\int_{\mathtt{Q}_{\delta,s}(z')}\sum_{i=1}^{n}a^{ii}\partial_i(v-\bar v)\partial_i\Big((v-\bar v)\xi ^2\Big)\\
=&-\sum_{j=i}^{n-1}\int_{\mathtt{Q}_{\delta,s}(z')}(F^i-\bar F^i)\partial_{i}\Big((v-\bar v)\xi ^2\Big)
-\sum_{j=1}^{n-1}\int_{\mathtt{Q}_{\delta,s}(z')}a^{nj}\partial_jv\partial_n(v-\bar v)\xi^2\\
&+\frac{1}{2\delta(z')}\int_{\mathtt{Q}_{\delta,s}(z')}\psi(v-\bar v)\xi^2+\int_{\mathtt{\Gamma}^-_{\delta,s}(z')}\psi(v-\bar v)\xi^2,
\end{split}
\end{equation}
here we used
$$-\int_{\mathtt{\Gamma}^-_{\delta,s}(z')}\Big(a^{nn}\partial_{n}(v-\bar{v})+\sum_{j=1}^{n-1}a^{nj}\partial_{j}v\Big)(v-\bar v)\xi^2=\int_{\mathtt{\Gamma}^-_{\delta,s}(z')}\psi(v-\bar v)\xi^2.$$

For the left hand side of \eqref{integratebyparts}, by using \eqref{aii}, we have 
\begin{align*}\label{ellipticity}
\int_{\mathtt{Q}_{\delta,s}(z')}\sum_{i=1}^{n}a^{ii}\partial_i(v-\bar v)\partial_i(v-\bar v)\xi ^2\geq\lambda\int_{\mathtt{Q}_{\delta,s}(z')}|\nabla(v-\bar v)|^2\xi^2.
\end{align*}
By virtue of the Cauchy inequality and the following inequality
\begin{equation}\label{poincare}
\int_{\mathtt{Q}_{\delta,s}(z')}|v-\bar v|^2\leq\,C\delta(z')^{2}\int_{\mathtt{Q}_{\delta,s}(z')}|\nabla(v-\bar v)|^2,
\end{equation}
we have
\begin{align*}
&\Big|\int_{\mathtt{Q}_{\delta,s}(z')}\sum_{i=1}^{n}2\xi (v-\bar v)a^{ii}\partial_i(v-\bar v)\partial_i\xi\Big| \nonumber\\
\leq&\, \frac{\lambda}{16}\int_{\mathtt{Q}_{\delta,s}(z')}|\nabla(v-\bar v)|^2\xi^2+C\int_{\mathtt{Q}_{\delta,s}(z')}|v-\bar v|^2|\nabla\xi|^2 \nonumber\\
\leq&\,\frac{\lambda}{16}\int_{\mathtt{Q}_{\delta,s}(z')}|\nabla(v-\bar v)|^2\xi^2+C\frac{\delta(z')^2}{(s-t)^2}\int_{\mathtt{Q}_{\delta,s}(z')}|\nabla(v-\bar v)|^2.
\end{align*}

For the first term on the right hand side of \eqref{integratebyparts}, using \eqref{F-barF} and the H\"older inequality leads to
\begin{align*}
\int_{\mathtt{Q}_{\delta,s}(z')}|F^i-\bar F^i|&|\nabla_i(v-\bar v)|\xi ^2
\leq\, C \delta(z')^{1/2}\int_{\mathtt{Q}_{\delta,s}(z')}|\nabla(v-\bar v)|^2\xi^2\nonumber\\
&+C \int_{\mathtt{Q}_{\delta,s}(z')}\left(\int_{-\delta(z')}^{\delta(z')}|\nabla(v-\bar v)|^2dy_n \right)^{1/2} |\nabla(v-\bar v)|\xi^2.
\end{align*}
Further, by using the H\"older inequality,
\begin{align*}
&\int_{\mathtt{Q}_{\delta,s}(z')}\left(\int_{-\delta(z')}^{\delta(z')}|\nabla(v-\bar v)|^2dy_n \right)^{1/2} |\nabla(v-\bar v)|\xi^2 \nonumber\\
\leq&\, C\delta(z')^{-1/2}\int_{\mathtt{Q}_{\delta,s}(z')} \left(\int_{-\delta(z')}^{\delta(z')}|\nabla(v-\bar v)|^2dy_n\right)\xi^2 +C\delta(z')^{1/2}\int_{\mathtt{Q}_{\delta,s}(z')}|\nabla(v-\bar v)|^2\xi^2\nonumber\\
\leq&\, C\delta(z')^{1/2}\int_{\mathtt{Q}_{\delta,s}(z')}|\nabla(v-\bar v)|^2\xi^2.
\end{align*}
For another term,
\begin{align*}
\int_{\mathtt{Q}_{\delta,s}(z')}|F^i-\bar F^i|&|v-\bar v|\xi |\partial_{i}\xi|
\leq\, C\int_{\mathtt{Q}_{\delta,s}(z')}|F-\bar F|^2\xi^2+C\int_{\mathtt{Q}_{\delta,s}(z')}|v-\bar v|^2|\nabla\xi|^2\nonumber\\
\leq&\, C \delta(z')\int_{\mathtt{Q}_{\delta,s}(z')}|\nabla(v-\bar v)|^2\xi^2+\frac{C\delta(z')^2}{(s-t)^2}\int_{\mathtt{Q}_{\delta,s}(z')}|\nabla(v-\bar v)|^2.
\end{align*}

While, for the second term on the right hand side of \eqref{integratebyparts}, by virtue of \eqref{ani1}, 
\begin{align*}
&\Big|\int_{\mathtt{Q}_{\delta,s}(z')}\sum_{j=1}^{n-1}a^{nj}\partial_jv\partial_n(v-\bar v)\xi^2\Big| \nonumber\\
\leq&\, \Big|\int_{\mathtt{Q}_{\delta,s}(z')}\sum_{j=1}^{n-1}a^{nj}\partial_j(v-\bar v)\partial_n(v-\bar v)\xi^2\Big|+\Big|\int_{\mathtt{Q}_{\delta,s}(z')}\sum_{j=1}^{n-1}a^{nj}\partial_j\bar v\partial_n(v-\bar v)\xi^2\Big|\nonumber\\
\leq&\, \Big(C\delta(z')^{1/2}+\frac{\lambda}{16}\Big)\int_{\mathtt{Q}_{\delta,s}(z')}|\nabla(v-\bar v)|^2\xi^2+C \delta(z')\int_{\mathtt{Q}_{\delta,s}(z')}|\nabla\bar v|^2\xi^2.
\end{align*}
For the third term, by means of \eqref{poincare},
\begin{align*}
\Big|\frac{1}{2\delta(z')}\int_{\mathtt{Q}_{\delta,s}(z')}\psi(v-\bar v)\xi^2\Big|\leq \frac{\lambda}{16}\int_{\mathtt{Q}_{\delta,s}(z')}|\nabla(v-\bar v)|^2\xi^2+C\int_{\mathtt{ \mathtt{Q}_{\delta,s}(z')}}\psi^2\xi^2.
\end{align*}
For the fourth term, choosing a function $0\leq\zeta(y_n)\leq1$ such that $\zeta(\delta(z'))=0$, $\zeta(-\delta(z'))=1$, and $|\nabla\zeta(y_n)|\leq\frac{4}{\delta(z')}$, and using \eqref{poincare} again, we obtain
\begin{align*}
\Big|\int_{\mathtt{\Gamma}^{-}_{\delta,s}(z)}\psi(v-\bar v)\xi^2\Big|=&\,\Big|\int_{\mathtt{\Gamma}^{-}_{\delta,s}(z)}\int_{-\delta(z')}^{\delta(z')}\psi(y')\xi^2\partial_n((v-\bar v)\zeta)dy_ndS\Big|\\
\leq&\,\frac{\lambda}{16}\int_{\mathtt{Q}_{\delta,s}(z')}|\nabla(v-\bar v)|^2\xi^2+C\int_{\mathtt{Q}_{\delta,s}(z')}\psi^2\xi^2.
\end{align*}
Substituting these estimates above into \eqref{integratebyparts} leads to
\begin{equation}\label{estimatevbarv}
\begin{split}
\int_{\mathtt{Q}_{\delta,s}(z')}&|\nabla(v-\bar v)|^2\xi^2 
\leq\, \Big(\frac{1}{4}+C \delta(z')^{1/2}\Big)\int_{\mathtt{Q}_{\delta,s}(z')}|\nabla(v-\bar v)|^2\xi^2 \\
&+\frac{C\delta(z')^2}{(s-t)^2}\int_{\mathtt{Q}_{\delta,s}(z')}|\nabla(v-\bar v)|^2+C\delta(z') \int_{\mathtt{Q}_{\delta,s}(z')}|\nabla\bar v|^2+C\int_{\mathtt{Q}_{\delta,s}(z')}\psi^2\xi^2.
\end{split}
\end{equation}

We now fix the constant $C$ in \eqref{estimatevbarv} to be $\bar{C}$, and then there exist a universal constant $R_{1}< R$, such that $\bar{C} \delta(z')^{1/2}<\frac{1}{4}$, if $|z'|<R_{1}$. (When $\delta(z')^{1/2}\geq\frac{1}{4\bar{C} }$, we can apply the standard elliptic theory.) Then \eqref{estimatevbarv} can be rewritten as 
\begin{equation}\label{iteration0}
\begin{split}
\int_{\mathtt{Q}_{\delta,s}(z')}|\nabla(v-\bar v)|^2\xi^2 \leq&\,\frac{2\bar{C}\delta(z')^{2}}{(s-t)^2}\int_{\mathtt{Q}_{\delta,s}(z')}|\nabla(v-\bar v)|^2 \\
&+2\bar{C}\Big(\delta(z') \int_{\mathtt{Q}_{\delta,s}(z')}|\nabla\bar v|^2+ \int_{\mathtt{Q}_{\delta,s}(z')}\psi^2\xi^2\Big).
\end{split}
\end{equation}
We choose $$t_0=\frac{\eta(z')}{8},\quad t_i=\frac{\eta(z')}{8}+2\sqrt{\bar{C}}i\eta(z'),\quad\,i=1,2,\dots,k,$$
such that
$$|t_{i}-t_{i-1}|^{2}=4\bar{C}\eta(z')^{2},$$
and denote $G(t_{i}):=\int_{\mathtt{Q}_{\delta,t_{i}}(z')}|\nabla(v-\bar v)|^2$. We have  the following iteration formula from \eqref{iteration0},  
\begin{equation}\label{iteration2}
\begin{split}
G(t_{i-1}) \leq&\frac{1}{2}G(t_{i})+2\bar{C}\Big(\delta(z') \int_{\mathtt{Q}_{\delta,t_i}(z')}|\nabla\bar v|^2+ \int_{\mathtt{Q}_{\delta,t_i}(z')}\psi^2\xi^2\Big).
\end{split}
\end{equation}

Applying this iteration formula $k(z')$ times, where $k(z')=\left[\frac{\frac{\sqrt{\eta(z')}}{64} -\frac{\eta(z')}{8}}{2\sqrt{\bar{C}}\eta(z')}\right]$, yields
\begin{align*}
G(t_0)\leq&\,(\frac{1}{2})^{k(z')}G(t_{k(z')})
+C_1\delta(z')^n\Big(\delta(z') \|\nabla\bar v\|_{L^{\infty}(B'_{\sqrt{\eta(z')}/64}(z'))}^2+\|\psi\|^2_{L^{\infty}(\Gamma^{-}_{\delta,  \sqrt{\eta(z')}/64 }(z))}\Big),
\end{align*}
where $C_1=4\bar{C}S_n\frac{\delta(z')^{n-1}}{\delta(z')^{n-1}}\sum_{i=1}^{k(z')}\frac{1}{2^i}\big(1+2(i+1)\sqrt{\bar{C}}\big)^{n-1}$ is another universal constant, and $S_n$ is the volume of unit sphere in $\R^{n-1}$. Choosing $\bar R <R_{1}$ such that $(\frac{1}{2})^{k(z')}\leq\eta(z')^n$ if $|z'|<\bar R$, we have
\begin{align*}
G(t_0)
\leq&\, \eta(z')^n\Big(G(t_{k(z')})+C_1\eta(z') \|\nabla\bar v\|_{L^{\infty}(B'_{\sqrt{\eta(z')}/64}(z'))}^2+C_1\|\psi\|^2_{L^{\infty}(\Gamma^{-}_{\eta, \sqrt{\eta(z')}/64}(z))}\Big).
\end{align*}
 By the definition of $\bar v$,  
$$\int_{\mathtt{Q}_{\delta,\sqrt{\eta(z')}/64}(z')}|\nabla\bar v|^2\leq\int_{\mathtt{Q}_{\delta, \sqrt{\eta(z')}/64}(z)}|\nabla v|^2.$$
So that
\begin{align*}
G(t_{k(z')})=\int_{\mathtt{Q}_{\delta,t_{k(z')}}(z')}|\nabla(v-\bar v)|^2\leq&\, C\int_{\mathtt{Q}_{\delta, \sqrt{\eta(z')}/64}(z)}|\nabla v|^2.
\end{align*}
Thus, the proof of Lemma \ref{lem_v-barv} is completed.
\end{proof}

\subsection{$L^{\infty}$ estimate of $\nabla\bar v$}

For $\nabla\bar v$, we have the following estimate.
 
\begin{lemma}\label{lem34}
Let $\bar v$ be the solution to \eqref{barv}. Then
\begin{equation}\label{estbarvv1}
\begin{split}
\|\nabla\bar v\|_{L^{\infty}(B'_{\sqrt{\eta(z')}/64}(z'))}\leq& C\eta(z')^{\frac{1-\mu}{2}}\|\nabla v\|_{C^{\mu}(\mathtt{Q}_{\delta,\sqrt{\eta(z')}/32}(z))}^* \\
&+C\eta(z')^{-1/2}\Big(\underset{\mathtt{Q}_{\delta,\sqrt{\eta(z')}/32}(z)}\osc~v+ C\|\psi\|_{L^{\infty}(\Gamma^{-}_{\delta,\sqrt{\eta(z')}/32}(z))}\Big),
\end{split}
\end{equation}
where $\mu$ is the same as Lemma \ref{lem_estv}. 
\end{lemma}

\begin{proof}
We rewrite \eqref{barv} in $\R^{n-1}$ as
$$\partial_i(a^{ii}\partial_i\bar v)=-\partial_i\bar{F}^i+G,\quad\mbox{where}~~ G(y'):=\frac{\psi(y')}{2\delta(z')}.$$  
 We set 
$$\bar{v}_{\sqrt{\eta}}(w'):=\bar v(z'+\frac{\sqrt{\eta(z')}}{32}w'),\quad |w'|\leq1,$$
to rescale  $\bar v$ in $B'_{ \sqrt{\eta(z')}/32}(z')$. Then $\bar{v}_{\sqrt{\eta}}$ satisfies 
$$\partial_i(a_{\sqrt{\eta}}^{ii}\partial_i\bar{v}_{\sqrt{\eta}})=-\partial_i\bar{F}_{\sqrt{\eta}}^i+G_{\sqrt{\eta}},$$
where 
$$a_{\sqrt{\eta}}^{ii}(w')=a^{ii}(z'+ \frac{\sqrt{\eta(z')}}{32}w'),\quad
\bar{F}_{\sqrt{\eta}}^i(w')=\frac{\sqrt{\eta(z')}}{32}\bar{F}^i(z'+\frac{\sqrt{\eta(z')}}{32}w'),$$
and$$\quad G_{\sqrt{\eta}}(w')=\frac{\ \eta(z')}{32^2}G(z'+ \frac{\sqrt{\eta(z')}}{32}w').$$

By the definition of $\bar{F}^i$, \eqref{ani1} and \eqref{ani2}, we have 
\begin{align*}
&[\bar{F}^i]_{C^{\mu}(B'_{\sqrt{\eta(z')}/32}(z'))}\leq C [a^{in}\partial_nv]_{C^{\mu}(\mathtt{Q}_{\delta,\sqrt{\eta(z')}/32}(z))}\\
\leq&\, C \eta(z')^{1/2}[\partial_nv]_{C^{\mu}(\mathtt{Q}_{\delta,\sqrt{\eta(z')}/32}(z))}+\eta(z')^{\frac{1-2\mu}{2}}\|\partial_nv\|_{L^{\infty}(\mathtt{Q}_{\delta,\sqrt{\eta(z')}/32}(z))}\\
\leq&\, C \eta(z')^{\frac{1-2\mu}{2}}\|\nabla v\|^{*}_{C^{\mu}(\mathtt{Q}_{\delta,\sqrt{\eta(z')}/32}(z))}.
\end{align*}
Hence
\begin{align*}
&[\bar{F}_{\sqrt{\eta}}^i]_{C^{\mu}(B'_{1}(0'))}\leq C \eta(z')^{\frac{1+\mu}{2}}[\bar{F}^i]_{C^{\mu}(B'_{\sqrt{\eta(z')}/32}(z'))}\leq C \eta(z')^{1-\frac{\mu}{2}}\|\nabla v\|^{*}_{C^{\mu}(\mathtt{Q}_{\delta,\sqrt{\eta(z')}/32}(z))}, 
\end{align*}
and
$$\|G_{\sqrt{\eta}}\|_{L^{\infty}(B'_{1}(0'))}\leq C\|\psi\|_{L^{\infty}(\Omega_{\sqrt{\eta(z')}/32}(z))}.$$

Since \eqref{aii}, it follows that $\lambda\leq\,a_{\sqrt{\eta}}^{ii}(w')\leq\Lambda$, for $|w'|\leq1$, and $[a_{\sqrt{\eta}}^{ii}]_{C^{\mu}(Q_1)}\leq C$. By the standard elliptic theory as before,  
\begin{align}\label{Calpha_est}
&\|\nabla(\bar{v}_{\sqrt{\eta}}-a)\|_{C^{\mu}(B'_{1/2}(0'))}\nonumber\\
\leq&\, C\left(\|(\bar{v}_{\sqrt{\eta}}-a)\|_{L^{2}(B'_{1}(0'))}+[\bar{F}_{\sqrt{\eta}}^i]_{C^{\mu}(B'_{1}(0'))}+\|G_{\sqrt{\eta}}\|_{L^{\infty}(B'_{1}(0'))}\right).
\end{align}
Taking $a=\underset{\mathtt{Q}_{\delta,\sqrt{\eta(z')}/32}(z)}\inf~v$, and by rescaling, 
\begin{align}\label{a}
&\sqrt{\eta(z')} \|\nabla\bar{v}\|_{L^{\infty}(B'_{\sqrt{\eta(z')}/64}(z'))}+\sqrt{\eta(z')}^{1+\mu} [\nabla\bar{v}]_{C^{\mu}(B'_{\sqrt{\eta(z')}/64}(z'))} \nonumber\\
\leq&\, C\Big(\,\underset{\mathtt{Q}_{\delta,\sqrt{\eta(z')}/32}(z)}\osc~v+\eta(z')^{1-{\mu}/{2}}\|\nabla v\|^{*}_{C^{\mu}(\mathtt{Q}_{\delta,\sqrt{\eta(z')}/32}(z))}+\|\psi\|_{L^{\infty}(\Gamma^{-}_{\delta,\sqrt{\eta(z')}/32}(z))}\Big).
\end{align}
 This implies that  \eqref{estbarvv1} holds.
\end{proof}

\subsection{Proof of Proposition \ref{prop_mainthm}}

 By using Lemma \ref{lem_v-barv} and
 $$\|\nabla \bar v\|_{L^2(\mathtt{Q}_{\delta,\eta(z')/8}(z))}\leq\,C\eta(z')^{n/2}\|\nabla \bar v\|_{L^{\infty}(\mathtt{Q}_{\delta,\eta(z')/8}(z))},$$
 we have
\begin{align*}
&\|\nabla v\|_{L^2(\mathtt{Q}_{\delta,\eta(z')/8}(z))}\nonumber\\
\leq &\,\|\nabla (v-\bar v)\|_{L^2(\mathtt{Q}_{\delta,\eta(z')/8}(z))}+\|\nabla \bar v\|_{L^2(\mathtt{Q}_{\delta,\eta(z')/8}(z))}\\
\leq&\,C\eta(z')^{n/2}\Big(\|\nabla v\|_{L^{2}(\mathtt{Q}_{\delta, \sqrt{\eta(z')}/64}(z))} +\|\nabla\bar v\|_{L^{\infty}(B'_{ \sqrt{\eta(z')}/64}(z'))}+\|\psi\|_{L^{\infty}(\mathtt{\Gamma}^{-}_{\delta,  \sqrt{\eta(z')}/64}(z))}\Big).
\end{align*}
Then, by Lemma \ref{lem_estv},
 \begin{align}\label{vbarv}
 \begin{split}
 &\|\nabla v\|_{C^{\mu}(\mathtt{Q}_{\delta, \eta(z')/16}(z))}^*\\
\leq &\,  C\eta(z')^{-n/2}\|\nabla v\|_{L^2(\mathtt{Q}_{\delta,\eta(z')/8}(z))}+C \|\psi\|_{C^{\alpha}(\mathtt{\Gamma}^{-}_{\delta, \eta(z')/8}(z))}^*\\
\leq &\,C\Big(\|\nabla v\|_{L^{2}(\mathtt{Q}_{\delta, \sqrt{\eta(z')}/64}(z))}  +\|\nabla\bar v\|_{L^{\infty}(B'_{ \sqrt{\eta(z')}/64}(z'))}+\|\psi\|_{C^{\alpha}(\mathtt{\Gamma}^{-}_{\delta,  \sqrt{\eta(z')}/64}(z))}^* \Big).
\end{split}
\end{align}

For the first term on the right hand side of \eqref{vbarv}, we have 
\begin{lemma}\label{lem35}
\begin{align}\label{lem35-1}
\|\nabla v\|_{L^{2}(\mathtt{Q}_{\delta, \sqrt{\eta(z')}/64}(z))} \leq
\,C\Big(\underset{\Omega_{ \sqrt{\eta(z')}/32}(z)}\osc\,u+\|\phi\|^{*}_{C^{\alpha}(\Gamma^{-}_{\delta,  \sqrt{\eta(z')}/32}(z))}\Big).
\end{align}
\end{lemma}

\begin{proof}
Since in cylinder $\mathtt{Q}_{\delta, \sqrt{\eta(z')}/64}(z)$, $|\nabla v(y)|\leq C |\nabla u(x)|$, and $|det_x y|\leq C,$ it follows that, for $|z'|<\bar{R}$,  
\begin{align}\label{lem35-2}
\int_{\mathtt{Q}_{\delta, \sqrt{\eta(z')}/64}(z)}|\nabla v(y)|^2dy
\leq\, C\int_{\Omega_{ \sqrt{\eta(z')}/64}(z)}|\nabla u(x)|^2dx.
\end{align}
Since $\Delta u=0$ in $\Omega_{ \sqrt{\eta(z')}/64}(z)$, it follows that
\begin{align*}
&\int_{\Omega_{ \sqrt{\eta(z')}/64}(z)}|\nabla u(x)|^2=\int_{\Omega_{ \sqrt{\eta(z')}/64}(z)}|\nabla (u(x)-\inf_{\Omega_{ \sqrt{\eta(z')}/64}(z)} u)|^2\\
=&\int_{\partial\Omega_{ \sqrt{\eta(z')}/64}(z)}\frac{\partial u}{\partial\nu}(x)\Big(u(x)-\inf_{\Omega_{ \sqrt{\eta(z')}/64}(z)} u\Big)\\
\leq&\,C(\underset{\Omega_{ \sqrt{\eta(z')}/64}(z)}\osc\,u)\Big(\int_{\Gamma^{-}_{ \sqrt{\eta(z')}/64}(z)}|\phi|+\int_{\{|x'-z'|= \sqrt{\eta(z')}/64\}\cap\Omega}|\nabla_{x'}u|\Big).
\end{align*}
For $|x'-z'|\leq\sqrt{\eta(z')}/64$, in order to apply Lemma \ref{lem_estv}, we set $\Omega_{ \eta(x')/8}(x)\subset\Omega_{ \sqrt{\eta(z')}/32}(z)$ for sufficiently small $z$, then by means of \eqref{lem1-1},  
\begin{align*}
|\nabla_{x'}u(x)|\leq&\,C\Big(\eta(x')^{-1}\underset{\Omega_{ \eta(x')/8}(x)}\osc\,u+\|\phi\|^*_{C^{\alpha}(\Gamma^{-}_{ \eta(x')/8}(x))}\Big)\\
\leq&\,C\Big(\eta(z')^{-1}\underset{\Omega_{ \sqrt{\eta(z')}/32}(z)}\osc\,u+\|\phi\|^*_{C^{\alpha}(\Gamma^{-}_{ \sqrt{\eta(z')}/32}(z))}\Big).
\end{align*}
Hence,
\begin{align*}
&\int_{\{|x'-z'|= \sqrt{\eta(z')}/64\}\cap\Omega}|\nabla_{x'}u(x)|\\\leq&\,C(\underset{\Omega_{ \sqrt{\eta(z')}/32}(z)}\osc\,u)\int_{\{|x'-z'|= \sqrt{\eta(z')}/64\}\cap\Omega}\eta(z')^{-1}dS+C\|\phi\|^*_{C^{\alpha}(\Gamma^{-}_{ \sqrt{\eta(z')}/32}(z))}\\
\leq&\,C\Big(\underset{\Omega_{ \sqrt{\eta(z')}/32}(z)}\osc\,u+\|\phi\|^*_{C^{\alpha}(\Gamma^{-}_{ \sqrt{\eta(z')}/32}(z))}\Big).
\end{align*}
Combining this with \eqref{lem35-2} implies that \eqref{lem35-1} holds.
\end{proof}

\begin{proof}[Proof of Proposition \ref{prop_mainthm}]
For $x\in\Omega_{\bar{R}}$, combining \eqref{lem35-1}, \eqref{vbarv} with \eqref{estbarvv1} in Lemma \ref{lem34}, and by virtue of Lemma \ref{uvsqrtdelta}, back to $u$, leads to
\begin{equation}\label{estu1}
\begin{split}
\|\nabla u\|_{C^{\mu}(\Omega_{ \eta(z')/16}(z))}^*\leq& C\eta(z')^{\frac{1-\mu}{2}}\|\nabla u\|_{C^{\mu}(\Omega_{\sqrt{\eta(z')}/32}(z))}^*\\
&+C\eta(z')^{-1/2}\Big(\underset{\Omega_{\sqrt{\eta(z')}/32}(z)}\osc\,u+ \|\phi\|^*_{C^{\alpha}(\Gamma^{-}_{\sqrt{\eta(z')}/32}(z))}\Big).
\end{split}
\end{equation}

Furthermore, it is clear that
\begin{equation}\label{sup0}
\|\nabla u\|_{C^{\mu}(\Omega_{\sqrt{\eta(x')}/32}(x))}^* \leq C \sup_{z\in\Omega_{\sqrt{\eta(x')}/32}(x)}\|\nabla u\|_{C^{\mu}(\Omega_{\eta(z')/16}(z))}^*.
\end{equation}
Then, substituting \eqref{estu1} into \eqref{sup0}, and using \eqref{ryz} yields
\begin{equation}\label{estu2}
\begin{split}
\|\nabla u&\|_{C^{\mu}(\Omega_{\sqrt{\eta(x')}/32}(x))}^*\leq C\eta(x')^{\frac{1-\mu}{2}}\sup_{z\in\Omega_{\sqrt{\eta(x')}/32}(x)}\|\nabla u\|_{C^{\mu}(\Omega_{ \sqrt{\eta(z')}/32}(z))}^*\\
&+C\eta(x')^{-1/2}\sup_{z\in\Omega_{\sqrt{\eta(x')}/32}(x)}\Big(\underset{\Omega_{ \sqrt{\eta(z')}/32}(z)}\osc\,u+ \|\phi\|^*_{C^{\alpha}(\Gamma^{-}_{ \sqrt{\eta(z')}/32}(z))}\Big).
\end{split}
\end{equation}
Putting the estimate of $\|\nabla u\|_{C^{\mu}(\Omega_{\sqrt{\eta(x')}/32}(x))}^*$ in \eqref{estu2} into \eqref{estu1} again, we have
\begin{equation}\label{estu3}
\begin{split}
\|\nabla u\|_{C^{\mu}(\Omega_{ \eta(x')/16}(x))}^*\leq& C\eta(x')^{1-\mu}\sup_{z\in\Omega_{\sqrt{\eta(x')}/32}(x)}\|\nabla u\|_{C^{\mu}(\Omega_{ \sqrt{\eta(z')}/32}(z))}^*\\
&+C\eta(x')^{-1/2}\sup_{z\in\Omega_{\sqrt{\eta(x')}/32}(x)}\Big(\underset{\Omega_{ \sqrt{\eta(z')}/32}(z)}\osc\,u+ \|\phi\|^*_{C^{\alpha}(\Gamma^{-}_{ \sqrt{\eta(z')}/32}(z))}\Big).
\end{split}
\end{equation}

By means of \eqref{ryz}, 
 $$\underset{z\in\Omega_{\sqrt{\eta(x')}/32}(x)}\cup\Omega_{ \sqrt{\eta(z')}/32}(z)\subset \Omega_{\frac{1}{8}\sqrt{\eta(x')}}(x).$$  
Then it follows from Remark \ref{u-1} that
\begin{equation}\label{estu4}
\begin{split}
 \sup_{z\in\Omega_{\sqrt{\eta(x')}/32}(x)}\|\nabla u\|_{C^{\mu}(\Omega_{ \sqrt{\eta(z')}/32}(z))}^*\leq& C\|\nabla u\|_{C^{\mu}(\Omega_{\frac{1}{8}\sqrt{\eta(x')}}(x))}^*\\
\leq&\, \eta(x')^{-1}\Big(\underset{\Omega_{\frac{1}{4}\sqrt{\eta(x')}}(x)}\osc\,u+ \|\phi\|^*_{C^{\alpha}(\Gamma^{-}_{\frac{1}{4}\sqrt{\eta(x')}}(x))}\Big).
\end{split}
\end{equation}
We substitute \eqref{estu4} into \eqref{estu3} to obtain
\begin{equation*}
\|\nabla u\|_{C^{\mu}(\Omega_{ \eta(x')/16}(x))}^*\leq C\Big(\eta(x')^{-\mu}+\eta(x')^{-{1}/{2}}\Big)\Big(\underset{\Omega_{\frac{1}{4}\sqrt{\eta(x')}}(x)}\osc\,u+ \|\phi\|^*_{C^{\alpha}(\Gamma^{-}_{\frac{1}{4}\sqrt{\eta(x')}}(x))}\Big).
\end{equation*}
The proof of Proposition \ref{prop_mainthm} is finished, because of the assumption $0<\mu\leq\frac{1}{2}$.
 
\end{proof}

 \section{Proof of Proposition \ref{ubound} }

In Section \ref{section3}, it has been shown that the estimate of $\nabla u$ is closely related with the estimate of $\nabla \bar u$. 
Recalling \eqref{barequ_v11} and using \eqref{barv}, $\bar u$ satisfies
\begin{equation}\label{barv2}
\sum_{i=1}^{n-1}\partial_i(\delta(x')\partial_i\bar u(x'))+\sum_{i=1}^{n-1}\partial_i\tilde F^{i}(x')= \psi(x'),\quad\mbox{in}~B'_{2R}(0'),
\end{equation}
where $\kappa(\varepsilon+|x'|^2)\leq\delta(x')\leq \frac{1}{\kappa}(\varepsilon+|x'|^2),$
and
\begin{equation}\label{DIVF}
\tilde F^i(x'):=\int_{g(x')}^{\varepsilon+f(x')}\Big(\frac{x_n-\varepsilon-f(x')}{\delta(x')}\partial_{x_i}g(x')-\frac{x_n-g(x')}{\delta(x')}\partial_{x_i}f(x')\Big)\partial_nu(x)dx_n.
\end{equation}
In this section, we first study elliptic equation with degenerate coefficients, which includes the type of \eqref{barv2}. Then combining  Proposition \ref{prop_mainthm} and the global and local properties of the solution to the degenerate elliptic equation, we can prove Proposition \ref{ubound}.

\subsection{Some estimates on the degenerate elliptic equations}

  For simplicity, we now introduce some notations.  For $  t \in \bR$ , we introduce the norm
$$
\| H \|_{\varepsilon, t,B'_{R}}: = \sup_{y' \in B'_{R}} \frac{|H(y')|}{ (\varepsilon + |y'|^2)^{t}}.
$$
We will adapt these notations throught this paper. 

\begin{prop}\label{zero}
For $n\geq3$, let $w\in H^1(B_{\rho})$ be a solution to 
\begin{equation}
\dv\Big[\big(\varepsilon I+A(x')\big)\nabla w(x')\Big]=\dv F+ G \quad\text{in}\,\,B_{\rho} \subset \bR^{n-1},
\end{equation}
where the $(n-1)\times(n-1)$ matrix $A(x')=(A^{ij}(x'))$ is measurable, symmetric and satisfies 
\begin{equation*}
\frac{1}{A}|x'|^2|\xi|^2\leq \xi^{T}A(x')\xi,\quad |A^{ij}(x')|\leq A|x'|^2 \quad\text{for}\,\, \forall \xi\in\bR^{n-1},\,\forall x'\in B_{\rho}
\end{equation*}
for some positive constant $A$.
If $F\in L^{\infty}(B_{\rho})$ and $G\in L^{\infty}(B_{\rho})$ satisfy
\begin{equation*}
\|F\|_{\varepsilon,\frac{\sigma+1}{2},B_{\rho}'}+\|G\|_{\varepsilon,\frac{\sigma }{2},B_{\rho}'}<\infty
\end{equation*}
for some $\sigma>0$, then we have
\begin{equation}
\|w\|_{L^{\infty}(B_{\rho}')}\leq \|w\|_{L^{\infty}(\partial B_{\rho}')}+\frac{C \rho^{\sigma}}{2^{\sigma}-1} (\|F\|_{\varepsilon,\frac{\sigma+1}{2},B_{\rho}'}+\|G\|_{\varepsilon,\frac{\sigma }{2},B_{\rho}'}),
\end{equation}
where the constant $C$ depends only on $A$ and $n$, and is in particular independent of $\varepsilon$ and $\rho$.
\end{prop}

If the hypothesis in Proposition \ref{zero} is futher weakened to allow $\sigma\geq0$, we can establish local oscillation estimates as follows.

\begin{prop}\label{nonzero}
For $n\geq3$, let $w\in H^1(B_{\rho})$ be a solution to 
\begin{equation}
\dv\Big[\big(\varepsilon I+A(x')\big)\nabla w(x')\Big]=\dv F+ G \quad\text{in}\,\,B_{\rho} \subset \bR^{n-1},
\end{equation}
where the $(n-1)\times(n-1)$ matrix $A(x')=(A^{ij}(x'))$ is measurable, symmetric and satisfies 
\begin{equation*}
\frac{1}{A}|x'|^2|\xi|^2\leq \xi^{T}A(x')\xi,\quad |A^{ij}(x')|\leq A|x'|^2 \quad\text{for}\,\, \forall \xi\in\bR^{n-1},\,\forall x'\in B_{\rho},
\end{equation*}
for some positive constant $A$.

(i) For $\sigma=0$, if $F\in L^{\infty}(B_{\rho})$ and $G\in L^{\infty}(B_{\rho})$ satisfy 
$$\|F\|_{\varepsilon,\frac{1}{2},B_{\rho}'}+\|G\|_{\varepsilon,0,B_{\rho}'}<\infty,$$ 
then for $x'\in B_{\frac{\rho}{4}}(0')$, we have
\begin{equation}
\underset{B_{\frac{1}{4}\sqrt{\eta(x')}}(x')}\osc\,w\leq C(\|w\|_{L^{\infty}(\partial B_{\rho}')}+\|F\|_{\varepsilon,\frac{1 }{2},B_{\rho}'}+\|G\|_{\varepsilon,0,B_{\rho}'})   ,
\end{equation}
where $C$ depends only on $A$ and $n$, but independent of $\varepsilon$ and $\rho$.

(ii) For $\sigma>0$, if $F\in L^{\infty}(B_{\rho})$ and $G\in L^{\infty}(B_{\rho})$ satisfy 
$$\|F\|_{\varepsilon, \frac{1+\sigma}{2},B_{\rho}'}+\|G\|_{\varepsilon, \frac{\sigma}{2},B_{\rho}'}<\infty,$$  
then for $x'\in B_{\frac{\rho}{4}}(0')$, we have
\begin{equation}
\underset{B_{\frac{1}{4}\sqrt{\eta(x')}}(x')}\osc\,w\leq C(\|w\|_{L^{\infty}(\partial B_{\rho}')}+\|F\|_{\varepsilon, \frac{1+\sigma}{2},B_{\rho}'}+\|G\|_{\varepsilon, \frac{\sigma}{2},B_{\rho}'})\left(\frac{\eta(x')^{\frac{\tau}{2}}}{\rho^{\tau}}\right)  ,
\end{equation}
where $\tau\in(0,1)$ and $C$ depend only on $A$, $n$ and the lower bound of $\sigma$, but independent of $\varepsilon$ and $\rho$.
\end{prop}
 
The proofs of Proposition \ref{zero} and Proposition \ref{nonzero}  are based on following two lemmas. 

\begin{lemma}\label{Linftyboundary}
For $n\geq3$, let $w\in H^1(B_{\rho})$ be a solution to 
\begin{equation}
\dv\Big[\big(\varepsilon I+A(x')\big)\nabla w(x')\Big]=\dv F+ G \quad\text{in}\,\,B_{\rho} \subset \bR^{n-1},
\end{equation}
where the $(n-1)\times(n-1)$ matrix $A(x')=(A^{ij}(x'))$ is measurable, symmetric and satisfies 
\begin{equation*}
\frac{1}{A}|x'|^2|\xi|^2\leq \xi^{T}A(x')\xi,\quad |A^{ij}(x')|\leq A|x'|^2 \quad\text{for}\,\, \forall \xi\in\bR^{n-1},\,\forall y'\in B_{\rho}
\end{equation*}
for some positive constant $A$.
If $F\in L^{\infty}(B_{\rho})$ and $G\in L^{\infty}(B_{\rho})$, then for $\rho>\sqrt\varepsilon$ we have
\begin{equation}
\|w\|_{L^{\infty}(B_{\rho}\backslash B_{\frac{1}{8}\rho})}\leq \|w\|_{L^{\infty}(\partial B_{\rho}')}+C\rho^{\sigma}(\|F\|_{\varepsilon, \frac{\sigma+1}{2},B_{\rho}'}+\|G\|_{\varepsilon,\frac{\sigma}{2},B_{\rho}'}),
\end{equation}
for some $\sigma\geq0$, where the constant $C$ depends only on $A$ and $n$, and is in particular independent of $\varepsilon$ and $\rho$.
\end{lemma}
\begin{proof}
Without loss of generality, we can assume that $$\|F\|_{\varepsilon,\frac{\sigma+1}{2},B_{\rho}'}+\|G\|_{\varepsilon,\frac{\sigma }{2},B_{\rho}'}\leq 1.$$ 
We decompose $w=w_1+w_2$ in $B_{\rho}(0'),$ where $w_2\in H_{0}^1(B_{\rho}(0'))$ satisfies 
\begin{equation}
\dv\Big[\big(\varepsilon I+A(y')\big)\nabla w_2(y')\Big]=\dv F+ G \quad\text{in}\,\,B_{\rho} \subset \bR^{n-1}.
\end{equation}
Then $w_1$ satisfies
\begin{equation}
\begin{cases}
\dv\Big[\big(\varepsilon I+A(y')\big)\nabla w_1(y')\Big]=0 &\quad\text{in}\,\,B_{\rho} \subset \bR^{n-1}\\
w_1(y')=w(y') &\quad\text{on}\,\,\partial B_{\rho} .
\end{cases}
\end{equation}
By the maximum principle, we have
\begin{equation}\label{w1}
\|w_1\|_{L^{\infty}(B_{\rho}(0'))}\leq \|w\|_{L^{\infty}(\partial B_{\rho}(0'))}.
\end{equation}
For $w_2$, we perform a change of variables by setting $z'=\frac{y'}{\rho}$, $\tilde w_2(z')=w_2(y')$, $\tilde F(z')=\frac{F(y')}{\rho}$ and $\tilde G(z') =G(y').$ Then $\tilde w_2\in H_0^1(B_1(0'))$ satisfies 
\begin{equation}\label{tildew2}
\dv\Big[\big(\frac{\varepsilon I+A(\rho z')}{\rho^2}\big)\nabla \tilde w_2(z')\Big]=\dv \tilde F+ \tilde G \quad\text{in}\,\,B_{1} \subset \bR^{n-1}.
\end{equation}
Multiplying $\tilde w_2$ in the equation \eqref{tildew2} and integrating by parts in equation \eqref{tildew2} yield,
\begin{equation}\label{424}
\begin{split}
&\quad\int_{B_1}\big(\frac{\varepsilon+\rho^2|z'|^2}{\rho^2}\big)|\nabla \tilde w_2(z')|^2 dz'\\
&\leq C\int_{B_1}\frac{|F(\rho z')|}{\rho}|\nabla \tilde w(z')|dz'+C\int_{B_1}|G(\rho z')\tilde w(z')|dz'\\
&\leq \beta\int_{B_1}\big(\frac{\varepsilon+\rho^2|z'|^2}{\rho^2}\big)|\nabla \tilde w_2(z')|^2 dz'+C_{\beta}\int_{B_1}\frac{|F(\rho z')|^2}{ \varepsilon+\rho^2|z'|^2}dz'\\
&\quad+\beta\int_{B_1}|z'|^{\frac{2(n-1)}{n+1}}|\tilde w_2(z')|^2 dz'+C_{\beta}\int_{B_1}|z'|^{-\frac{2(n-1)}{n+1}}|G(\rho z')|^2 dz'.
\end{split}
\end{equation}
Because $\tilde w_2\in H_0^1(B_1),$ by H\"older inequality and the following Caffarelli-Kohn-Nirenberg inequality in \cite{CKN} in $\bR^{n-1},$ $$\|w\|_{L^{\frac{2(n+1)}{n-1}}(B_1,|x'|^2dx')}\leq C\|\nabla w\|_{L^2(B_1,|x'|^2dx')}\quad \forall w\in H_0^1(B_1,|x'|^2dx') ,$$ we know that 
\begin{equation}\label{425}
\begin{split}
\int_{B_1}|z'|^{\frac{2(n-1)}{n+1}}|\tilde w_2(z')|^2dz'&\leq C(n)\left( \int_{B_1}|z'|^2|| \tilde w_2(z')|^{\frac{2(n+1)}{n-1}}dz'\right)^{\frac{n-1}{n+1}}\\
&\leq C(n)\int_{B_1}|z'|^2|\nabla \tilde w_2(z')|^2 dz'\\
&\leq C(n)\int_{B_1}(\frac{\varepsilon}{\rho^2}+|z'|^2)|\nabla \tilde w_2(z')|^2 dz'
\end{split}
\end{equation} 
Using \eqref{425}, choosing an appropriate $\beta$ in equation \eqref{424}, we have
\begin{equation}\label{426}
\begin{split}
\int_{B_1}\big(\frac{\varepsilon+\rho^2|z'|^2}{\rho^2}\big)|\nabla \tilde w_2(z')|^2 dz'\leq C \int_{B_1}\frac{|F(\rho z')|^2}{ \varepsilon+\rho^2|z'|^2}dz'+C\int_{B_1} |z'|^{-\frac{2(n-1)}{n+1}}|G(\rho z')|^2 dz'.
\end{split}
\end{equation}
Combining \eqref{426} and the fact that
\begin{equation}
\begin{split}
\int_{B_1}\frac{|F(\rho z')|^2}{ \varepsilon+\rho^2|z'|^2}dz'&\leq C\int_{B_1}(\varepsilon+\rho^2|z'|^2)^{\sigma}dz'\leq C\rho^{2\sigma},\\
\int_{B_1} |z'|^{-\frac{2(n-1)}{n+1}}|G(\rho z')|^2 dz'&\leq C\rho^{2\sigma}\int_{B_1} |z'|^{-\frac{2(n-1)}{n+1}}dz'\leq C \rho^{2\sigma},
\end{split}
\end{equation}
we can obtain
\begin{equation}
\int_{B_1\backslash B_{\frac{1}{16}}}|\nabla \tilde w_2(z')|^2dz'\leq C\rho^{2\sigma}.
\end{equation}
Using Poincar\'e inequality, we know
$$\int_{B_1\backslash B_{\frac{1}{16}}}|\tilde w_2(z')|^2dz'\leq C\int_{B_1\backslash B_{\frac{1}{16}}}|\nabla \tilde w_2(z')|^2dz'\leq C\rho^{2\sigma}.$$
Then by rescaling and making use of the local estimates at the boundary (see Theorem 8.25 in \cite{GT}), we can get
\begin{equation}\label{w2}
\begin{split}
\|w_2\|^2_{L^{\infty}(B_{\rho}\backslash B_{\frac{1}{8}\rho})}&=\|\tilde w_2\|^2_{L^{\infty}(B_{1}\backslash B_{\frac{1}{8} })}\leq C\left(\int_{B_1\backslash B_{\frac{1}{16}}}|\tilde w_2(z')|^2dz'+\rho^{2\sigma}\right)\leq C\rho^{2\sigma}.
\end{split}
\end{equation}
The proof is finished by \eqref{w1} and \eqref{w2}.

\end{proof}

\begin{lemma}\label{oscillationw}
For $n\geq3$, let $w_1\in H^1(B_{\rho})$ be a solution to 
\begin{equation}
\dv\Big[\big(\varepsilon I+A(x')\big)\nabla w_1(x')\Big]=0 \quad\text{in}\,\,B_{\rho}(0') \subset \bR^{n-1},
\end{equation}
where the $(n-1)\times(n-1)$ matrix $A(x')=(A^{ij}(x'))$ is measurable, symmetric and satisfies 
\begin{equation*}
\frac{1}{A}|x'|^2|\xi|^2\leq \xi^{T}A(x')\xi,\quad |A^{ij}(x')|\leq A|x'|^2 \quad\text{for}\,\, \forall \xi\in\bR^{n-1},\,\forall y'\in B_{\rho},
\end{equation*}
for some positive constant $A$. Then for $\rho>\sqrt{\varepsilon} $, we have
\begin{equation}
\underset{B_{\frac{\rho}{2}}}\osc\, w_1\leq \beta\underset{\partial B_{\rho}} \osc\, w_1,
\end{equation}
where $\beta\in(0,1)$ is a constant depending only on $A$ and $n$.
\end{lemma}
\begin{proof}
We perform a change of variables by setting $z'=\frac{x'}{\rho}$ and let $\tilde w_1(z')=w_1(x').$
Then $\tilde w_1$ satisfies 
\begin{equation}
\dv\left(\frac{\varepsilon I+A(\rho z')}{\rho^2}\nabla \tilde w_1(z')\right)=0\quad z'\in B_1\backslash B_{\frac{1}{8}}.
\end{equation}
Making use of Harnack inequality (see Corollary 8.21 in \cite{GT}), we have
\begin{equation}\label{equ433}
\begin{aligned}
\underset{z'\in B_{\frac{1}{2}}\backslash B_{\frac{1}{4}}}\sup(\tilde w_1(z')-\underset{B_{\rho}}\inf\, w_1)&\leq C\underset{z'\in B_{\frac{1}{2}}\backslash B_{\frac{1}{4}}}\inf(\tilde w_1(z')-\underset{B_{\rho}}\inf\, w_1),\\
\underset{z'\in B_{\frac{1}{2}}\backslash B_{\frac{1}{4}}}\sup(\underset{B_{\rho}}\sup\, w_1-\tilde w_1(z'))&\leq C\underset{z'\in B_{\frac{1}{2}}\backslash B_{\frac{1}{4}}}\inf(\underset{B_{\rho}}\sup\, w_1-\tilde w_1(z')).
\end{aligned}
\end{equation}
By \eqref{equ433} and rescaling, we can get
\begin{equation}\label{equ434}
\begin{aligned}
\underset{  B_{\frac{\rho}{2}}\backslash B_{\frac{\rho}{4}}}\sup\, w_1 -\underset{B_{\rho}}\inf\, w_1 &\leq C (\underset{  B_{\frac{\rho}{2}}\backslash B_{\frac{\rho}{4}}}\inf\, w_1 -\underset{B_{\rho}}\inf\, w_1 ) ,\\
\underset{B_{\rho}}\sup\, w_1-\underset{  B_{\frac{\rho}{2}}\backslash B_{\frac{\rho}{4}}}\inf w_1 &\leq C ( \underset{B_{\rho}}\sup\, w_1-\underset{  B_{\frac{\rho}{2}}\backslash B_{\frac{\rho}{4}}}\sup w_1  ).
\end{aligned}
\end{equation}
Adding up the above two inequlities in \eqref{equ434}, we know that,
\begin{equation}\label{435}
\underset{  B_{\frac{\rho}{2}}\backslash B_{\frac{\rho}{4}}}\osc\, w_1  \leq \frac{C-1}{C+1}\underset{B_{\rho}}\osc\, w_1. 
\end{equation}
By the weak maximum principle, we have
\begin{equation}\label{436}
\begin{split}
\underset{  B_{\frac{\rho}{2}}\backslash B_{\frac{\rho}{4}}}\osc\, w_1 \leq\underset{B_{\frac{\rho}{2}}}\osc\, w_1&=\underset{\partial B_{\frac{\rho}{2}}}\osc\, w_1\leq \underset{  B_{\frac{\rho}{2}}\backslash B_{\frac{\rho}{4}}}\osc\, w_1,\\
\underset{B_{\rho}}\osc\, w_1&=\underset{\partial B_{\rho}}\osc\, w_1.
\end{split}
\end{equation}
By \eqref{435} and \eqref{436}, we finish the proof.
\end{proof}
Now we are ready to prove Proposition \ref{zero} and Proposition \ref{nonzero}.
\begin{proof}[Proof of Proposition \ref{zero}]
Without loss of generality, we assume that $$\|F\|_{\varepsilon,\frac{\sigma+1}{2},B_{\rho}'}+\|G\|_{\varepsilon,\frac{\sigma }{2},B_{\rho}'}\leq 1.$$ 
It is well know that $w$ is locally H\"older continuous in $B_{\rho}'.$  For $\sqrt\varepsilon<|x'|<\frac{\rho}{4}$, we consider $w$ in $B_{2|x'|}$,
\begin{equation}
\dv\Big[\big(\varepsilon I+A(y')\big)\nabla w(y')\Big]=\dv F+ G \quad\text{in}\,\,B_{2|x'|} \subset \bR^{n-1}.
\end{equation}
Using Lemma \ref{Linftyboundary}, we can know
\begin{equation}
|w(x')|\leq \|w\|_{L^{\infty}(\partial B_{2|x'|})}+C(2|x'|)^{\sigma},
\end{equation}
which implies
\begin{equation}\label{Linftyiteration}
\|w\|_{L^{\infty}(\partial B_{|x'|})}\leq \|w\|_{L^{\infty}(\partial B_{2|x'|})}+C(2|x'|)^{\sigma}.
\end{equation}
Iterating \eqref{Linftyiteration} $k$ times, where $k$ satisfies $\frac{\rho}{4}\leq 2^{k-1}|x'|<2^k|x'|<\rho$, we have
\begin{equation}\label{414}
\begin{split}
\|w\|_{L^{\infty}(\partial B_{|x'|})}&\leq \|w\|_{L^{\infty}(\partial B_{2^k|x'|})}+C\sum_{i=1}^{k}2^{\sigma i}|x'|^{\sigma}\leq \|w\|_{L^{\infty}(B_{\rho}\backslash B_{\frac{\rho}{4}})}+C\frac{2^{(k+1)\sigma}|x'|^{\sigma}}{2^{\sigma}-1}.
\end{split}
\end{equation}
Noting that $2^{k\sigma}|x'|^{\sigma}<\rho^{\sigma}$, by \eqref{414} and Lemma \ref{Linftyboundary} we have proved that 
\begin{equation}\label{Linftyepsilon} 
\|w\|_{L^{\infty}(B_{\rho}\backslash B_{\sqrt\epsilon})}\leq \|w\|_{L^{\infty}(\partial B_{\rho})}+\frac{C \rho^{\sigma}}{2^{\sigma}-1}.
\end{equation}
For $|x'|\leq \sqrt\varepsilon$, considering $w$ in $B_{\sqrt\varepsilon}$, by a scaling argument, Theorem 8.16 in \cite{GT} and \eqref{Linftyepsilon}, we can obtain
\begin{equation}\label{Linftyepsilon1} 
\|w\|_{L^{\infty}( B_{\sqrt\epsilon})}\leq \|w\|_{L^{\infty}(\partial B_{\rho})}+\frac{C \rho^{\sigma}}{2^{\sigma}-1}.
\end{equation}
The proof is finished by \eqref{Linftyepsilon}, \eqref{Linftyepsilon1} and Lemma \ref{Linftyboundary}.
\end{proof}
\begin{proof}[Proof of Proposition \ref{nonzero} ]
In the proof, for $\sigma\geq 0$, we assume that$$\|w\|_{L^{\infty}(\partial B_{\rho}')}+\|F\|_{\varepsilon, \frac{1+\sigma}{2},B_{\rho}'}+\|G\|_{\varepsilon, \frac{\sigma}{2},B_{\rho}'}\leq1.$$
For $\sqrt{\varepsilon}<|x'|<\frac{\rho}{4}$, we decompose $w=w_1+w_2$ in $B_{2|x'|}(0'),$ where $w_2\in H_{0}^1(B_{2|x'|}(0'))$ satisfies 
\begin{equation}
\dv\Big[\big(\varepsilon I+A(y')\big)\nabla w_2(y')\Big]=\dv F+ G \quad\text{in}\,\,B_{2|x'|} \subset \bR^{n-1}.
\end{equation}
Then $w_1$ satisfies
\begin{equation}
\begin{cases}
\dv\Big[\big(\varepsilon I+A(y')\big)\nabla w_1(y')\Big]=0 &\quad\text{in}\,\,B_{2|x'|} \subset \bR^{n-1},\\
w_1(y')=w(y') &\quad\text{on}\,\,\partial B_{2|x'|} .
\end{cases}
\end{equation}
If $\sigma=0$, by Lemma \ref{Linftyboundary} and Lemma \ref{oscillationw}, for some $\beta\in(0,1)$, we have
\begin{equation}\label{iterationosc}
\begin{split}
\underset{B_{|x'|}\backslash B_{\frac{|x'|}{2}}}\osc\,w&\leq \underset{B_{|x'|}\backslash B_{\frac{|x'|}{2}}}\osc\,w_1+ \underset{B_{|x'|}\backslash B_{\frac{|x'|}{2}}}\osc\,w_2\\
&\leq \underset{B_{|x'|}}\osc\,w_1+\underset{B_{2|x'|}\backslash B_{\frac{|x'|}{2}}}\osc\,w_2\\
&\leq \beta \underset{\partial B_{2|x'|}}\osc\,w_1 + C \\
&\leq \beta\underset{B_{2|x'|}\backslash B_{|x'|}}\osc\,w+ C .
\end{split}
\end{equation}
Iterating \eqref{iterationosc} for $k$ times, where $k$ satisfies $\frac{\rho}{4}\leq 2^{k-1}|x'|<2^{k}|x'|<\rho,$ and using Lemma \ref{Linftyboundary}, we have
\begin{equation}\label{Linftyepsilon2}
\begin{split}
 \underset{B_{|x'|}\backslash B_{\frac{|x'|}{2}}}\osc\,w&\leq\beta^{k}\underset{B_{2^k|x'|}\backslash B_{2^{k-1}|x'|}}\osc\,w+C\sum_{i=0}^{k-1}\beta^{i}\leq \underset{B_{\rho}\backslash B_{\frac{\rho}{4}}}\osc\,w+C\leq C.
\end{split}
\end{equation}
For $|x'|\leq \sqrt\varepsilon$, considering $w$ in $B_{\sqrt\varepsilon}$, by a scaling argument, Theorem 8.16 in \cite{GT} and \eqref{Linftyepsilon2}, we can obtain
\begin{equation}\label{Linftyepsilon3} 
\underset{B_{\sqrt\varepsilon}(0')}\osc\,w\leq C.
\end{equation}
The proof of Proposition \ref{nonzero} (i) is finished by \eqref{Linftyepsilon2} and \eqref{Linftyepsilon3}.

If $\sigma>0$, By Lemma \ref{Linftyboundary} and Lemma \ref{oscillationw}, for some $\beta\in(0,1)$, we have
\begin{equation}\label{iterationosc1}
\begin{split}
\underset{B_{|x'|}\backslash B_{\frac{|x'|}{2}}}\osc\,w&\leq \underset{B_{|x'|}\backslash B_{\frac{|x'|}{2}}}\osc\,w_1+ \underset{B_{|x'|}\backslash B_{\frac{|x'|}{2}}}\osc\,w_2\\
&\leq \underset{B_{|x'|}}\osc\,w_1+\underset{B_{2|x'|}\backslash B_{\frac{|x'|}{2}}}\osc\,w_2\\
&\leq \beta \underset{\partial B_{2|x'|}}\osc\,w_1 + C|x'|^\sigma \\
&\leq \beta\underset{B_{2|x'|}\backslash B_{|x'|}}\osc\,w+ C|x'|^{\sigma} .
\end{split}
\end{equation}
We can make $\sigma>0$ smaller such that $\beta2^\sigma\neq1$. Iterating \eqref{iterationosc1} for $k$ times, where k satisfies $\frac{\rho}{4}\leq 2^{k-1}|x'|<2^{k}|x'|<\rho,$ taking $\beta=2^{-\tilde\beta}$ for $\tilde \beta>0$ and using Lemma \ref{Linftyboundary}, we have
\begin{equation}\label{Linftyepsilon4}
\begin{split}
 \underset{B_{|x'|}\backslash B_{\frac{|x'|}{2}}}\osc\,w&\leq\beta^{k}\underset{B_{2^k|x'|}\backslash B_{2^{k-1}|x'|}}\osc\,w+C\sum_{i=0}^{k-1}\beta^{i}2^{i\sigma}|x'|^{\sigma}\\
&\leq \beta^{k}\underset{B_{\rho}\backslash B_{\frac{\rho}{4}}}\osc\,w+C\frac{1-2^{k\sigma}\beta^{k}}{1-2^{\sigma}\beta}|x'|^{\sigma}\\
&\leq C\left(\frac{|x'|}{\rho}\right)^{\tilde\beta}+C|x'|^{\sigma}.
\end{split}
\end{equation}
Taking $\tau=\min\{\tilde\beta,\sigma\}$, we have proved the case for $\sqrt\varepsilon<|x'|<\frac{\rho}{4}$. For $|x'|\leq \sqrt\varepsilon$, considering $w$ in $B_{\sqrt\varepsilon}$, by a scaling argument, Theorem 8.16 in \cite{GT} and \eqref{Linftyepsilon4}, we can obtain
\begin{equation}\label{Linftyepsilon5} 
\underset{B_{\sqrt\varepsilon}(0')}\osc\,w\leq C\left(\frac{\sqrt\varepsilon}{\rho}\right)^{\tau}.
\end{equation}
The proof of Proposition \ref{nonzero} (ii) is finished by \eqref{Linftyepsilon4} and \eqref{Linftyepsilon5}  .
\end{proof}

\subsection{Proof of Proposition \ref{ubound}}
\begin{proof}[Proof of Proposition \ref{ubound}]
Without loss of generality, we can assume that $\phi(0')\geq 0$ and $\fint_{\Omega\backslash\Omega_{\tilde{R}/2}}u=0$.  
Denote $$\mathcal{U}_{\tilde{R}}:=\Big\|u-\fint_{\Omega\backslash\Omega_{\tilde{R}/2}}u\Big\|_{L^{\infty}({\Omega}\backslash\Omega_{3\tilde{R}/4})}=\left\|u \right\|_{L^{\infty}({\Omega}\backslash\Omega_{3\tilde{R}/4})} ,$$
 and
$$\mathcal{O}_{\tilde{R}}:=\sup_{x\in\Omega_{\tilde{R}}}\underset{\Omega_{\frac{1}{4}\sqrt{\eta(x')}}(x)}\osc\,u,$$ 
where the constant $\tilde{R}$ will be determined later, which is at least smaller than $\bar{R}$. In the following, we use $C_{0}$ to denote the constant dependent of $\bar{R}$ in the Proposition \ref{prop_mainthm}, but independent of $\varepsilon$ and $\tilde{R}$. $C$ is a constant dependent of $\tilde R$, but independent of $\varepsilon$.

By Proposition \ref{prop_mainthm}, we have 
\begin{equation}\label{estgradientu}
|\nabla_{x'}u(x)|\leq \left(C_0\mathcal{O}_{\tilde{R}} +C\|\phi\|_{C^{\alpha}(\partial D)}\right)\eta(x')^{-1/2},\quad\mbox{for}~|x'|\leq \tilde{R},
\end{equation}
where $C_{0}$ is the constant in Proposition \ref{prop_mainthm}. Recall $\partial_{\nu}u=0$ on $\partial{D}_{1}$ and $\partial_{\nu}u=\phi$ on $\partial{D}$. Then,
on $\Gamma_{\tilde R}^{\pm}$,  
\begin{equation}\label{estgradientun1}
|\partial_nu(x)|\leq C_0|x'||\nabla_{x'}u(x')|+|\phi(x)|\leq C_0 \mathcal{O}_{\tilde{R}} +C\|\phi\|_{C^{\alpha}(\partial D)};
\end{equation}
On $\partial \Omega\backslash\Gamma_{\tilde R}^{\pm}$, by using the $W^{2,p}$ estimate and the Sobolev embedding theorem, 
\begin{equation}\label{estgradientun2}
|\partial_n u(x)|\leq \|\nabla u\|_{L^{\infty}(\Omega\backslash\Omega_{\tilde R})}\leq C \mathcal{U}_{\tilde{R}}+C\|\phi\|_{C^{\alpha}(\partial D)}.
\end{equation}
Since $\partial_nu$ is harmonic in $\Omega$, by applying the maximum principle, 
\begin{equation}\label{estgradientu2}
|\partial_n u(x)|\leq C_0\mathcal{O}_{\tilde{R}}+C\mathcal{U}_{\tilde{R}}+C\|\phi\|_{C^{\alpha}(\partial D)},\quad\mbox{in}~\Omega.
\end{equation}

Recalling \eqref{barv2}, we decompose $\bar u=u_1+u_2$, where $u_1$ verifies 
\begin{equation}
\begin{cases}
\dv(\delta(y')\nabla u_1)=-\dv \tilde{F}+\psi(y')-\psi(0')\quad &\mbox{in}\, B'_{\tilde{R}},\\
u_1=\bar u\quad&\mbox{on}\, \partial B'_{\tilde{R}},
\end{cases}
\end{equation}
while $u_2$ verifies 
\begin{equation}
\begin{cases}
\dv(\delta(y')\nabla u_2)=\psi(0')\quad &\mbox{in}\, B'_{\tilde{R}},\\
u_2=0\quad&\mbox{on}\, \partial B'_{\tilde{R}}.
\end{cases}
\end{equation}
By \eqref{DIVF}, \eqref{estgradientu2} and $\psi(y')=\phi(y')\sqrt{1+|\nabla_{y'}g(y')|^2}$, $\phi\in C^{\alpha}(\partial D)$, we have 
\begin{equation}\label{estFG}
\begin{split}
& \|\tilde F  \|_{\varepsilon, \frac{3}{2},B'_{\tilde{R}}}\leq \, C_0\mathcal{O}_{\tilde{R}}+C\mathcal{U}_{\tilde{R}}+C\|\phi\|_{C^{\alpha}(\partial D)},\\
&\| \psi -\psi(0')\|_{\varepsilon,\frac{\alpha}{2},B'_{\tilde{R}}}\leq \, C\|\phi\|_{C^{\alpha}(\partial D)} .
\end{split}
\end{equation}
By applying Proposition \ref{zero}, we have
\begin{equation}\label{normv1}
\begin{aligned}
\|u_1\|_{L^{\infty}(B'_{\tilde{R}}(0'))}&\leq\|u\|_{L^{\infty}(\partial B'_{\tilde{R}}(0'))}+C_0\tilde R^2 \|\tilde F  \|_{\varepsilon, \frac{3}{2},B'_{\tilde{R}}}+C\|\psi -\psi(0')\|_{\varepsilon,  \frac{\alpha}{2},B'_{\tilde{R}}}\\
&\leq C_0\tilde R^2\mathcal{O}_{\tilde{R}}+C\mathcal{U}_{\tilde{R}}+C\|\phi\|_{C^{\alpha}(\partial D)}.
\end{aligned}
\end{equation}
For $u_2$, making use of Prposition \ref{nonzero}, for $x'\in B_{\frac{\tilde R}{4}}(0')$, we have
\begin{equation}\label{oscu2}
\underset{{B}'_{\frac{1}{4}\sqrt{\eta(x')}}(x')}\osc~ u_2\leq C|\phi(0')|,
\end{equation}
and by Lemma \ref{Linftyboundary}, for $x'\in B_{\tilde R}\backslash B_{\frac{\tilde R}{8}} $, we have
\begin{equation}\label{oscu22}
\underset{B_{\tilde R}\backslash B_{\frac{\tilde R}{8}}}\osc~ u_2\leq C|\phi(0')|.
\end{equation}
Morever, by $M\geq\Delta \delta(y')\geq \frac{1}{M}$, setting $\tilde \delta_1(y')=M\psi(0')\ln\delta(y')$ and $\tilde \delta_2(y')=\frac{1}{M}\psi(0')\ln\delta(y')$, we can get 
\begin{equation}
\dv\left(\delta(y')\nabla\tilde \delta_2(y')\right)\leq \dv(\delta(y')\nabla u_2(y'))\leq \dv\left(\delta(y')\nabla\tilde \delta_1(y')\right).
\end{equation}
Thus, by maximun principle, we have
\begin{equation}\label{blowupu2}
\begin{split}
&u_2(y')\geq M\psi(0')\ln\delta(y')-C\psi(0')|\ln \tilde R|,\\
&  u_2(y')\leq\frac{1}{M}\psi(0')\ln\delta(y')+C\psi(0')|\ln \tilde R|.
\end{split}
\end{equation}
 By \eqref{normv1} and \eqref{oscu2}, for $x\in\Omega_{\frac{\tilde{R}}{4}}$, we have
\begin{align}\label{equ222}
\underset{{B}'_{\frac{1}{4}\sqrt{\eta(x')}}(x')}\osc~ \bar u 
\leq&\, 2\|u_1\|_{L^{\infty}({B}'_{\frac{1}{4}\sqrt{\eta(x')}}(x'))}+\underset{{B}'_{\frac{1}{4}\sqrt{\eta(x')}}(x')}\osc~ u_2  \nonumber\\
\leq&\,  C_0\tilde{R}^{2}\mathcal{O}_{\tilde{R}}+C\mathcal{U}_{\tilde{R}}+C\|\phi\|_{C^{\alpha}(\partial D)}.
\end{align}
By \eqref{normv1} and \eqref{oscu22}, for $x\in\Omega_{\tilde R}\backslash \Omega_{\frac{\tilde R}{4}}$, we can get
\begin{equation}\label{equ2222}
\begin{split}
\underset{{B}'_{\frac{1}{4}\sqrt{\eta(x')}}(x')}\osc~ \bar u&\leq \underset{B_{2\tilde R}\backslash B_{ \frac{\tilde R}{8}}}\osc~ \bar u\\
&\leq \underset{B_{\tilde R}\backslash B_{ \frac{\tilde R}{8}}}\osc~ \bar u+\underset{B_{2\tilde R}\backslash B_{\tilde R}}\osc~ \bar u\\
&\leq 2\|u_1\|_{L^{\infty}(B_{\tilde R}\backslash B_{ \frac{\tilde R}{8}})} +\underset{B_{\tilde R}\backslash B_{ \frac{\tilde R}{8}}}\osc~ u_2+2\mathcal{U}_{\tilde R}\\
&\leq  C_0\tilde{R}^{2}\mathcal{O}_{\tilde{R}}+C\mathcal{U}_{\tilde{R}}+C\|\phi\|_{C^{\alpha}(\partial D)}.
\end{split}
\end{equation}

On the other hand, 
$$\underset{\Omega_{\frac{1}{4}\sqrt{\eta(x')}}(x)}\osc~(u-\bar u)\leq2\|u-\bar u\|_{L^{\infty}(\Omega_{\frac{1}{4}\sqrt{\eta(x')}}(x))}\leq\,C_0\eta(x')\|\partial_n u\|_{L^{\infty}(\Omega_{\frac{1}{4}\sqrt{\eta(x')}}(x'))}.$$
Thus, by virtue of \eqref{estgradientu2},\eqref{equ222} and \eqref{equ2222}, for $x\in\Omega_{\tilde{R}}$,
\begin{align*}
\underset{\Omega_{\frac{1}{4}\sqrt{\eta(x')}}(x)}\osc\,u\leq&\,\, \underset{\Omega_{\frac{1}{4}\sqrt{\eta(x')}}(x)}\osc~(u-\bar u)+\underset{B'_{\frac{1}{4}\sqrt{\eta(x')}}(x')}\osc~ \bar u \\
\leq&\, C_0\tilde{R}^2\|\partial_n u\|_{L^{\infty}(\Omega_{\tilde{R}})}+ C_0\tilde{R}^{2}\mathcal{O}_{\tilde{R}}+C\mathcal{U}_{\tilde{R}}+C\|\phi\|_{C^{\alpha}(\partial D)}\\
\leq&\, C_0\tilde{R}^2(C_0\mathcal{O}_{\tilde{R}}+C\mathcal{U}_{\tilde{R}}+C\|\phi\|_{C^{\alpha}(\partial D)})+C_0\tilde{R}^{2}\mathcal{O}_{\tilde{R}}+C\mathcal{U}_{\tilde{R}}+C\|\phi\|_{C^{\alpha}(\partial D)}\\
\leq&\, C_0\tilde{R}^{2}\mathcal{O}_{\tilde{R}}+C\mathcal{U}_{\tilde{R}}+C\|\phi\|_{C^{\alpha}(\partial D)}. 
\end{align*}
Recalling the definition of $\mathcal{O}_{\tilde{R}}$, we have
\begin{equation*}
\mathcal{O}_{\tilde{R}}\leq C_0\tilde{R}^{2}\mathcal{O}_{\tilde{R}}+C\mathcal{U}_{\tilde{R}}+C\|\phi\|_{C^{\alpha}(\partial D)}.
\end{equation*}
Noting that $C_0$ is a constant independent of $\tilde R$ and $\varepsilon$, we can choose a sufficiently small $\tilde{R}$, such that $C_0\tilde{R}^{2}=\frac{1}{2}$, then we have
\begin{equation}\label{UROR}
\mathcal{O}_{\tilde{R}}\leq C\mathcal{U}_{\tilde{R}}+C\|\phi\|_{C^{\alpha}(\partial D)}.
\end{equation}
Therefore, by means of \eqref{normv1},
\begin{equation}\label{estu222}
\|u_1\|_{L^{\infty}(B'_{\tilde{R}}(0'))}\leq C\mathcal{U}_{\tilde{R}}+C\|\phi\|_{C^{\alpha}(\partial D)}.
\end{equation}
Now we fix this $\tilde{R}$.

By using \eqref{blowupu2}, \eqref{estu222}, 
 \begin{align} \label{nablauL2255} 
\int_{\Gamma_{\tilde{R}}^{-}}|\bar u|^2
\leq&\, C\int_{B'_{\tilde{R}}(0')}|\bar u|^2 \nonumber\\
\leq&\, C\int_{B'_{\tilde{R}}(0')}|u_1|^2+C\int_{B'_{\tilde{R}}(0')}|u_2|^2 \nonumber\\
\leq&\, \Big(C\mathcal{U}_{\tilde{R}}+C\|\phi\|_{C^{\alpha}(\partial D)}\Big)^{2}\tilde{R}^{n-1}+ C|\phi(0')|^2\int_{B'_{\tilde{R}}(0')}|\ln(\epsilon+|x'|^2)|^2 \nonumber\\
\leq&\, C \mathcal{U}_{\tilde{R}}^2+C\|\phi\|_{C^{\alpha}(\partial D)}^2,
\end{align}
while, by \eqref{UROR},
 \begin{align} \label{nablauL221} 
 \int_{\Gamma_{\tilde{R}}^{-}}|u-\bar u|^2
\leq\, C \mathcal{O}_{\tilde{R}}^2
\leq&\, C \mathcal{U}_{\tilde{R}}^2+C\|\phi\|_{C^{\alpha}(\partial D)}^2,
\end{align}
then we have
 \begin{align} \label{nablauL22} 
 \int_{\Gamma_{\tilde{R}}^{-}}u^2
\leq \,C\int_{\Gamma_{\tilde{R}}^{-}}|u-\bar u|^2+C \int_{\Gamma_{\tilde{R}}^{-}}|\bar u|^2 \leq\, C \mathcal{U}_{\tilde{R}}^2+C\|\phi\|_{C^{\alpha}(\partial D)}^2.
\end{align}
Recalling the assumption $\fint_{\Omega\backslash\Omega_{\tilde{R}/2}}u=0$, and using the trace theorem and the Poincar\'{e} inequality, we have
 \begin{align} \label{nablauL2}
 \int_{\partial{D}\setminus\Gamma_{\tilde{R}/2}^{-}}u^2\leq\,C\int_{\Omega\backslash\Omega_{\tilde{R}/2}}(u^{2}+|\nabla u|^2)\leq\,C\int_{\Omega\backslash\Omega_{\tilde{R}/2}} |\nabla u|^2.
\end{align}

By a similar bootstrap argument as in the proof of Lemma \ref{lem_estv}, we have
\begin{align} \label{nablauLinfty}
\mathcal{U}_{\tilde{R}}\leq&\,C\Big\|u-\fint_{\Omega\setminus\Omega_{\tilde{R}/2}}u\Big\|_{L^{2}(\Omega\setminus\Omega_{\tilde{R}/2})}+C\|\phi\|_{C^{\alpha}(\partial D)}\nonumber\\
\leq&\,C\|\nabla u\|_{L^{2}(\Omega\setminus\Omega_{\tilde{R}/2})}+C\|\phi\|_{C^{\alpha}(\partial D)}.
\end{align} 
Substituting this in \eqref{nablauL22} yields
$$\int_{\Gamma_{\tilde{R}}^{-}}u^2\leq\, C\|\nabla u\|^{2}_{L^{2}(\Omega\setminus\Omega_{\tilde{R}/2})}+C\|\phi\|^{2}_{C^{\alpha}(\partial D)}.$$
Hence, together with \eqref{nablauL2},
$$\int_{\partial D}u^2\leq\int_{\Gamma_{\tilde{R}}^{-}}u^2+\int_{\partial{D}\setminus\Gamma_{\tilde{R}/2}^{-}}u^2\leq\, C\|\nabla u\|^{2}_{L^{2}(\Omega\setminus\Omega_{\tilde{R}/2})}+C\|\phi\|^{2}_{C^{\alpha}(\partial D)}.$$
By the Cauchy inequality,  
\begin{align} \label{nablauL20}
\int_{\Omega}|\nabla u|^2 =&\,\int_{\partial D}u\phi \leq \mu \int_{\partial D}u^2+\frac{C}{\mu}\int_{\partial D}|\phi|^2\nonumber\\
\leq &\,C\mu\int_{\Omega}|\nabla u|^2+\Big(\frac{C}{\mu}+C\mu\Big)\|\phi\|^{2}_{C^{\alpha}(\partial D)},
\end{align}
Then, choosing a sufficiently small $\mu$, such that $C\mu=\frac{1}{2}$, yields
$$\int_{\Omega}|\nabla u|^2\leq\,C\|\phi\|^{2}_{C^{\alpha}(\partial D)}.$$
Hence, by virtue of \eqref{nablauLinfty},
\begin{equation}\label{estu2223}
\mathcal{U}_{\tilde{R}}\leq\,C\|\phi\|_{C^{\alpha}(\partial D)},
\end{equation}
and by \eqref{UROR}, we have
\begin{equation}\label{estu2224}
\mathcal{O}_{\tilde{R}}\leq\,C\|\phi\|_{C^{\alpha}(\partial D)}.
\end{equation}
This implies that \eqref{uinfty} and \eqref{oscu} hold.
\end{proof}

\section{Proof of Theorem \ref{cor1}, Corollary \ref{cor11} and Theorem \ref{lowerbound}}

In this section, we prove  Theorem \ref{cor1}, Corollary \ref{cor11} and Theorem \ref{lowerbound}. 

\subsection{Proof of Theorem \ref{cor1}}

\begin{proof}[Proof of Theorem \ref{cor1}]
By using Proposition \ref{prop_mainthm},  Proposition \ref{ubound}, it is obvious that \eqref{upperbd} holds. 

For $n<p<\frac{2}{1-\alpha}$, using the boundary estimates (Theorem 6.27 in \cite{L}) in the region $\Omega\setminus\Omega_{\frac{3}{4}\tilde{R}}$, and Proposition \ref{ubound} again, we have
\begin{align}\label{nablau_infty}
\|\nabla u\|_{L^{\infty}(\Omega\setminus\Omega_{\tilde{R}})}\leq&\,C\Big\|u-\fint_{\Omega\setminus\Omega_{\frac{1}{2}\tilde{R}}}u\Big\|_{W^{2,p}(\Omega\setminus\Omega_{\frac{7}{8}\tilde{R}})}\nonumber\\
\leq&\,\,C\Big\|u-\fint_{\Omega\setminus\Omega_{\frac{1}{2}\tilde{R}}}u\Big\|_{L^{p}(\Omega\setminus\Omega_{\frac{3}{4}\tilde{R}})}+C\|\phi\|_{1-\frac{1}{p},p;\partial{\Omega}\setminus\Gamma_{\frac{3}{4}\tilde{R}}}\nonumber\\
\leq&\,C\|\phi\|_{C^{\alpha}(\partial D)}.
\end{align}
Recall $\partial_{\nu}u=0$ on $\partial{D}_{1}$ and $\partial_{\nu}u=\phi$ on $\partial{D}$. On $\Gamma^\pm_{\tilde{R}}$, by virtue of \eqref{upperbd},
$$|\partial_{n}u(x)|\leq\,C|x'||\partial_{x'}u(x)|+C\|\phi\|_{C^{\alpha}(\partial{\Omega})}\leq\,C\|\phi\|_{C^{\alpha}(\partial{\Omega})}.$$
On $\partial\Omega\setminus\Gamma^\pm_{\tilde{R}}$, by \eqref{nablau_infty}, 
$$|\partial_{n}u(x)|\leq\|\nabla u\|_{L^{\infty}(\Omega\setminus\Omega_{\tilde{R}})}\leq\,C\|\phi\|_{C^{\alpha}(\partial{\Omega})}.$$
Since $\partial_{n}u$ is harmonic, applying the maximum principle, the proof is completed.
\end{proof}

\subsection{Proof of Corollary \ref{cor11}}

\begin{proof}[Proof of Corollary \ref{cor11} (i)]
Without loss of generality, we assume that $\fint_{\Omega\backslash\Omega_{\tilde R/2}} u=0$. By virtue of \eqref{normv1}, \eqref{estu2223} and \eqref{estu2224}, we have
\begin{equation}\label{normv112}
\|u_1\|_{L^{\infty}(B'_{\tilde{R}}(0'))}\leq\,C\|\phi\|_{C^{\alpha}(\partial D)}, 
\end{equation}
By \eqref{blowupu2},
\begin{equation}\label{normv12}
\| u_2\|_{L^{\infty}(B'_{\tilde{R}}(0'))}\leq\,C|\phi(0')||\ln\varepsilon|+C\|\phi\|_{C^{\alpha}(\partial D)}.
\end{equation}

  If $\phi(0')\neq0$, by \eqref{blowupu2} and \eqref{normv112}, we have
$$|\bar u(0')|\geq |u_2(0')|-|u_1(0')|\geq C|\phi(0')||\ln\epsilon|-C\|\phi\|_{C^{\alpha}(\partial D)}.$$
Thus, \eqref{blowuposcu} is proved. 

On the other hand, since for $x\in\Omega_{\tilde{R}}$, $|u(x)-\bar u(x)|\leq\mathcal{O}_{\tilde{R}}\leq\,C\|\phi\|_{C^{\alpha}(\partial D)}$, by using \eqref{normv112} and \eqref{normv12}, we have
\begin{equation}\label{ubdd}
\begin{split}
|u(x)|\leq&\, |\bar u(x')|+|u(x)-\bar u(x')|\\
\leq&\,|u_1(x')|+ |u_2(x') |+C\|\phi\|_{C^{\alpha}(\partial D)}\\
\leq&\, C|\phi(0')||\ln\epsilon|+C\|\phi\|_{C^{\alpha}(\partial D)}.
\end{split}
\end{equation}
So, \eqref{blowuposcu1} holds.
\end{proof}

\begin{proof}[Proof of Corollary \ref{cor11} (ii)]
We have shown that $\bar u$ is the solution to \eqref{barv2}, where $\tilde F$ is defined as \eqref{DIVF}.

By $\|\partial_n u\|_{L^{\infty}(\Omega)}\leq C\|\phi\|_{C^{\alpha}(\partial D)}$ in Theorem \ref{cor1}, $\psi(y')=\sqrt{1+|\nabla_{y'}g(y')|^2}\phi(y')$, $\phi(0')=0$ and \eqref{ubdd}, we have
\begin{equation}
\begin{split}
|\tilde F(x')|&\leq C\|\phi\|_{C^{\alpha}(\partial D)}(\varepsilon+|x'|^2)^{\frac{3}{2}},\\
|\psi(x')|&\leq C\|\phi\|_{C^{\alpha}(\partial D)}|x'|^{\alpha},\\
\|\bar{u}-\bar{u}(0)\|_{L^{\infty}(\Omega_{\tilde{R}})}&\leq\,C\|\phi\|_{C^{\alpha}(\partial D)}.
\end{split}
\end{equation}
Using  Proposition \ref{nonzero} (ii), for $x\in\Omega_{\frac{R}{4}}$, we know there exists a universal constant $\tilde \alpha\in(0,1)$ such that
\begin{equation}\label{zero_oscbaru}
\underset{B_{\frac{1}{4}\sqrt{\eta(x')}(x')}}\osc\,\bar u\leq C\|\phi\|_{C^{\alpha}(\partial D)}\eta(x')^{\frac{\tilde \alpha}{2}}.
\end{equation}
Using $\|\partial_n u\|_{L^{\infty}(\Omega)}\leq C\|\phi\|_{C^{\alpha}(\partial D)}$ in Theorem \ref{cor1} again, by \eqref{zero_oscbaru}, for $x\in\Omega_{\frac{R}{4}}$, we know
\begin{equation}\label{zero_oscu}
\underset{\Omega_{\frac{1}{4}\sqrt{\eta(x')}(x')}}\osc\,  u\leq C\|\phi\|_{C^{\alpha}(\partial D)}\eta(x')^{\frac{\tilde \alpha}{2}}.
\end{equation}
The proof is finished by Proposition \ref{prop_mainthm} and \eqref{zero_oscu}.

\end{proof}

\subsection{Proof of Theorem \ref{lowerbound}}

\begin{proof}[Proof of Theorem \ref{lowerbound}]
Without loss of generality, we assume that $\int_{\Omega\backslash \Omega_{\tilde R/2}}u=0$. Now we consider following auxiliary
$$\tilde u(x')=\bar u(x')-\frac{\phi(0')}{\sum_{i=1}^{n-1}b^{ii}(0')}\ln \delta(x'),$$
where $\bar u$ is the solution to \eqref{barv2} and $b^{ii}(x')=\partial_{ij}\delta(x')=\partial_{ij}(f-g)(x')$. Then by a direct computation, we know $\tilde u$ is the solution to
\begin{equation}
\sum_{i=1}^{n-1}\partial_i(\delta(x')\partial_i\tilde u(x'))=-\sum_{i=1}^{n-1}\partial_i\tilde F^{i}(x')+\psi(x')- \frac{\sum_{i=1}^{n-1}b^{ii}(x')}{\sum_{i=1}^{n-1}b^{ii}(0')}\phi(0').
\end{equation}
By $\|\partial_n u\|_{L^{\infty}(\Omega)}\leq C\|\phi\|_{C^{\alpha}(\partial D)}$ in Theorem \ref{cor1} and the definition of $\tilde F$ in \eqref{DIVF}, we have
\begin{equation}\label{lower1} 
|\tilde F(x')|\leq C\|\phi\|_{C^{\alpha}(\partial D)}(\varepsilon+|x'|^2)^{\frac{3}{2}}.
\end{equation}
Because $\phi\in C^{\alpha}$ and $b^{ii}\in C^{\gamma}$, for $\psi(x')- \frac{\sum_{i=1}^{n-1}b^{ii}(x')}{\sum_{i=1}^{n-1}b^{ii}(0')}\phi(0')$, we have
\begin{equation}\label{lower2}
\begin{split}
\left|\psi(x')- \frac{\sum_{i=1}^{n-1}b^{ii}(x')}{\sum_{i=1}^{n-1}b^{ii}(0')}\phi(0')\right|&\leq \left|\left(\phi(x')-\phi(0')\right)\sqrt{1+|\nabla_{x'}g(x')|^2}\right|\\
&+\left|\phi(0')\sqrt{1+|\nabla_{x'}g(x')|^2}-\phi(0')\right|+\left|\phi(0')\left(1- \frac{\sum_{i=1}^{n-1}b^{ii}(x')}{\sum_{i=1}^{n-1}b^{ii}(0')}\right)\right|\\
&\leq C\|\phi\|_{C^{\alpha}(\partial D)}|x'|^{\alpha}+C|\phi(0')||x'|^{\gamma}\\
&\leq C\|\phi\|_{C^{\alpha}(\partial D)}|x'|^{\alpha\gamma}
\end{split}
\end{equation}
By \eqref{lower1} and \eqref{lower2}, we can get
\begin{equation}\label{lower3}
\|\tilde F\|_{\varepsilon,\frac{3}{2},B_{\tilde R}}+\left\|\psi(x') - \frac{\sum_{i=1}^{n-1}b^{ii}(x') }{\sum_{i=1}^{n-1}b^{ii}(0')}\phi(0')\right\|_{\varepsilon,\frac{\alpha\gamma}{2},B_{\tilde R}}\leq C \|\phi\|_{C^{\alpha}(\partial D)}.
\end{equation}
By Proposition \ref{ubound}, \eqref{lower3} and  Proposition \ref{nonzero} (ii), for $\tilde u$, we have 
\begin{equation}
\underset{B_{ \sqrt{\varepsilon}/4 }'}\osc\,\tilde u\leq C\|\phi\|_{C^{\alpha}(\partial D)}|\ln \varepsilon| \varepsilon^{\frac{\tau}{2}} \tilde R^{-\tau},
\end{equation}
where $\tau\in(0,1)$ is a universal constant. So for sufficiently small $\varepsilon$ such that
$$\|\phi\|_{C^{\alpha}(\partial D)}|\ln \varepsilon| \varepsilon^{\frac{\tau}{2}} \tilde R^{-\tau}\leq\frac{1}{2C^2}|\phi(0')|,$$
we have
\begin{equation}\label{lowerr1}
\underset{B_{ \sqrt{\varepsilon}/4 }'}\osc\,\tilde u \leq \frac{1}{2C}|\phi(0')|.
\end{equation}
Because $\delta(x')=\varepsilon+(f-g)(x')\geq \varepsilon+\kappa|x'|^2$ and $\delta(0')=\varepsilon$, where $\kappa>0$, we have
\begin{equation}
\underset{x'\in B_{ \sqrt{\varepsilon}/4 }'}\sup \left(\delta(x')/\delta(0')\right)\geq 1+\frac{\kappa}{16}.
\end{equation}

This implies
\begin{equation}\label{lowerr2}
\underset{B_{ \sqrt{\varepsilon}/4 }'} \osc\, \tilde u_1\geq \frac{1}{C}|\phi(0')|,
\end{equation} 
where $\tilde u_1:=\frac{\phi(0')}{\sum_{i=1}^{n-1}b^{ii}(0')}\ln \delta(x')$.
By \eqref{lowerr1} and \eqref{lowerr2}, we know that
$$\underset{B_{ \sqrt{\varepsilon}/4 }'} \osc\,\bar u\geq\underset{B_{ \sqrt{\varepsilon}/4 }'} \osc\,\tilde u_1-\underset{B_{ \sqrt{\varepsilon}/4 }'} \osc\,\tilde u\geq\frac{1}{2C}|\phi(0')|. $$
Then by the mean value theorem, the proof is finished.

\end{proof}

\section{Proof of Theorem \ref{thDir}}
\begin{proof}[Proof of Theorem \ref{thDir}]
Without loss of generality, we assume that $\|\varphi\|_{C^{1,\alpha}(\partial\Omega)}\leq1.$ We extend $\varphi$ to domain $\Omega$ such that $\|\nabla\varphi\|_{C^{1,\alpha}(\Omega)}\leq\,C$. We only need to focus on the narrow region $\Omega_{R}$. 
Set $$w(x):=u(x)-\varphi(x),$$
then $w$ satisfies
\begin{equation*}\label{equinftyw01}
\begin{cases}
\dv(\nabla w)=-\dv(\nabla \varphi) & \mbox{in}~ \Omega_{R},\\
\partial_{\nu}w=-\partial_{\nu}\varphi&\mbox{on}~\Gamma_{R}^{+},\\
 w=0&\mbox{on}~\Gamma_{R}^{-},
\end{cases}
\end{equation*}
and $w+\varphi$ satisfies
\begin{equation}\label{equinftyw011}
\begin{cases}
\dv(\nabla (w+\varphi))= 0 & \mbox{in}~ \Omega_{R},\\
\partial_{\nu}(w+\varphi)=0&\mbox{on}~\Gamma_{R}^{+},\\
 w+\varphi=\varphi&\mbox{on}~\Gamma_{R}^{-}.\\
\end{cases}
\end{equation}

For $|z'|<R$, $0<t<s<\frac{1}{8}\sqrt{\eta(z')}$, let $\xi$ be a cutoff function satisfying $\xi(y')=1$ if $|y'-z'|<t$, $\xi(y')=0$ if $|y'-z'|>s$, $0\leq\xi(x')\leq1$ if $t\leq|x'-z'|\leq s$, and $|\nabla_{x'}\xi(x')|\leq\frac{2}{s-t}$. Multiplying the equation in \eqref{equinftyw011} by $  w\xi^{2}$, and by using integration by parts, we have
\begin{equation*}
\int_{\Omega_s(z)}\nabla(w+\varphi)\cdot\nabla (w\xi^2)=0.
\end{equation*}
So that
\begin{equation*}
\int_{\Omega_s(z)}|\nabla w|^2\xi^2\leq C\int_{\Omega_s(z)}w^2|\nabla \xi|^2+ C \int_{\Omega_s(z)}|\nabla \varphi|^2\xi^2.
\end{equation*}
Note that
\begin{equation*}
\begin{aligned}
\int_{\Omega_s(z)}w^2|\nabla\xi|^2\leq&\, \frac{C}{(s-t)^2}\int_{\Omega_s(z)}w^2(x)dx\\
\leq&\, \frac{C}{(s-t)^2}\int_{\Omega_s(z)}\Big(\int_{(x',-\epsilon/2+g(x'))}^x\partial_nw(x)dx_n\Big)^2dx\\
\leq&\, \frac{C\eta(z')^2}{(s-t)^2}\int_{\Omega_s(z)}|\nabla w|^2 dx,
\end{aligned}
\end{equation*}
and $\|\nabla\varphi\|_{L^{\infty}(\Omega)}\leq\,C$. Thus, 
\begin{equation}\label{iteration1}
\int_{\Omega_t(z)}|\nabla w|^2 \leq  \frac{C_0\eta(z')^2}{(s-t)^2}\int_{\Omega_s(z)}|\nabla w|^2+ C s^{n-1}\eta(z').
\end{equation}

Denote $$G(t)=\int_{\Omega_t(z)}|\nabla w|^2,\quad\mbox{and} ~t_0=\eta(z'),~ t_{i+1}=t_i+2\sqrt{C_0}\eta(z').$$
Then, by \eqref{iteration1}, we have $$G(t_i)\leq\frac{1}{4}G(t_{i+1})+	Ct_{i+1}^{n-1}\eta(z').$$
Choosing $k(z)= [\frac{1}{\sqrt{\eta(z')}}]$, similar as in the proof of Lemma \eqref{lem_v-barv}, after iterating $k(z)$ times, we have 
\begin{equation*} 
\int_{\Omega_{\eta(z')}(z)}|\nabla w|^2 \leq  C\eta(z')^n.
 \end{equation*}
This implies that $\int_{\Omega_{\eta(z')}(z)}|\nabla u|^2 \leq  C\eta(z')^n
$ as well. By a rescaling technique and the $W^{2,p}$ estimates and a bootstrap argument as in \cite{BLL,L}, we finish the proof.
\end{proof}

\noindent{\bf Acknowledgements.} The work of H. Li was partially Supported by Beijing Natural Science Foundation (No.1242006), the Fundamental Research Funds for the Central Universities (No.2233200015), and National Natural Science Foundation of China (No.12471191).

\noindent{\bf Conflict of Interest.} The authors declare that they have no conflict of interest.


\end{document}